\documentclass[a4paper,11pt]{amsart}

\usepackage{a4wide}
\newtheorem{prop}{Proposition}
\newtheorem{lemma}{Lemma}
\newtheorem{rem}{Remark}
\newtheorem{defi}{Definition}
\newtheorem{thm}{Theorem}
\usepackage{amssymb}
\usepackage{amsmath}
\usepackage{bm}
\numberwithin{equation}{section}
\numberwithin{prop}{section}
\numberwithin{lemma}{section}
\numberwithin{rem}{section}
\numberwithin{thm}{section}

\numberwithin{defi}{section}

\begin{document}

\title{Killing $2$-forms in dimension $4$}
\author{Paul Gauduchon and Andrei Moroianu}
\address{Paul Gauduchon \\ CMLS\\ {\'E}cole
  Polytechnique \\ UMR 7640 du CNRS
\\ 91128 Palaiseau \\ France}
\email{paul.gauduchon@polytechnique.edu}

\address{Andrei Moroianu \\ Universit\'e de Versailles-St Quentin \\
Laboratoire de Math\'ema\-tiques \\ UMR 8100 du CNRS\\
45 avenue des \'Etats-Unis\\
78035 Versailles, France }
\email{andrei.moroianu@math.cnrs.fr}
\date{\today}
\maketitle

\tableofcontents

\section{Introduction}

 On any $n$-dimensional Riemannian manifold $(M, g)$, an exterior $p$-form $\psi$ is called {\it conformal Killing}\footnote{Conformal Killing forms have the following conformal invariance property: if $\psi$ is a conformal Killing $p$-form with respect to the metric $g$, then, for any positive function $f$, $\tilde{\psi} := f ^{p + 1} \, \psi$ is conformal Killing with respect to the conformal metric $\tilde{g} := f ^2 \, g$. In other words, if $L$ denotes the real line bundle $|\Lambda ^n TM| ^{\frac{1}{n}}$ and $\ell, \tilde{\ell}$ denote 
the sections of $L$ determined by $g, \tilde{g}$,    then, for {\it any} Weyl connection $D$ relative to the conformal class $[g]$,  the section ${\bm \psi} := \psi \otimes \ell ^{p + 1} = \tilde{\psi} \otimes \tilde{\ell} ^{p + 1}$ of $\Lambda ^p T ^* M \otimes L ^{p + 1}$ satisfies
\begin{equation} D _X {\bm \psi} = {\bm \alpha} \wedge X + X \lrcorner {\bm \beta}, \end{equation}
for some section ${\bm \alpha}$ of $\Lambda ^{p - 1}  T^* M \otimes L ^{p - 1}$ and some section ${\bm \beta}$ of $\Lambda ^{p + 1} T^* M \otimes L ^{p + 1}$ (depending on $D$),  cf. {\it e.g.} \cite[Appendix B]{ACG2}.} \cite{uwe} if its covariant derivative $\nabla  \psi$  is of the form
\begin{equation} \label{conf-killing} \nabla  _X \psi = \alpha \wedge X ^{\flat} + X \lrcorner \beta, \end{equation}
for some $(p - 1)$-form $\alpha$ and some $(p + 1)$-form $\beta$, which are then given by
\begin{equation} \alpha = \frac{(- 1) ^p}{n - p + 1} \, \delta \psi, \qquad \beta = \frac{1}{p + 1} \, d \psi. \end{equation}

The $p$-form $\psi$ is called {\it Killing}, resp. $*$-{\it Killing}, with respect to $g$, if $\psi$ satisfies 
(\ref{conf-killing}) {\it and}  $\alpha = 0$, resp. $\beta = 0$. In particular, Killing forms are  co-closed, $*$-Killing forms are closed, and, if $M$ is oriented and $*$ denotes the induced Hodge star operator, $\psi$ is Killing if and only if $* \psi$ is $*$-Killing. 

\smallskip

\smallskip

Although the terminology comes from the fact that Killing $1$-forms are just metric duals of Killing vector fields, and thus encode infinitesimal symmetries of the metric, no geometric interpretation of Killing $p$-forms exists in general in terms of symmetries when  $p \geq 2$, except in the case of Killing $2$-forms in dimension $4$, which is special for various reasons, the most important  being the self-duality phenomenon. 

On any oriented four-dimensional manifold $(M, g)$, the Hodge star operator $*$, acting on $2$-forms, is an involution and, therefore, induces the well known orthogonal decomposition
\begin{equation} \label{split-lambda} \Lambda ^2 M = \Lambda ^+ M \oplus \Lambda ^- M, \end{equation}
where $\Lambda ^2 M$ stands for the vector bundle of (real) $2$-forms on $M$ and $\Lambda ^{\pm} M$ the eigen-subbundle for the eigenvalue $\pm 1$ of $*$. 
Accordingly, any $2$-form $\psi$ splits as 
\begin{equation} \label{split-psi} \psi = \psi _+ + \psi _-, \end{equation}
 where $\psi _+$, resp. $\psi _-$, is the {\it self-dual}, resp. the {\it anti-self-dual} part of $\psi$, defined by $\psi _{\pm}  = \frac{1}{2} (\psi \pm  * \psi)$. Since $*$ acting on $2$-forms is conformally invariant, a $2$-form 
$\psi$ is conformal Killing if and only if $\psi _+$ and $\psi _-$ are separately conformal Killing, meaning that 
\begin{equation} \label{conf-killing+} \nabla  \psi _+ = (\alpha _+ \wedge X ^{\flat}) _+, \qquad 
\nabla  \psi _- = (\alpha _- \wedge X ^{\flat}) _- \end{equation}
 for some real $1$-forms $\alpha _+, \alpha _-$, and $\psi$ is Killing, resp. $*$-Killing, if, in addition, 
\begin{equation} \label{coupling} \alpha _+ = -  \alpha _-, \quad \text{resp.} \quad  \alpha _+ =  \alpha _-. \end{equation}

\smallskip

Throughout  this paper, $(M, g)$ will denote  a connected,  oriented, $4$-dimen\-sional Riemannian  manifold and $\psi = \psi _+ + \psi _-$ a non-trivial $*$-Killing $2$-form on $M$  (the choice of the $*$-Killing $\psi$, instead of the  Killing  $2$-form $* \psi$ is of pure convenience). We also discard the non-interesting case when $\psi$ is parallel.

\smallskip

On the open set, $M _0 ^+$, resp. $M ^- _0$, where $\psi _+$, resp. $\psi _-$,  has no zero, the associated skew-symmetric operators $\Psi _+, \Psi _-$, are of the form $\Psi _+ = f _+ \, J _+$, resp. $\Psi _- = f _- \, J _-$, where $J _+$, resp. $J _-$,  is an almost complex structure inducing the chosen, resp. the opposite, orientation of $M$,  and $f _+$, resp. $f _-$,   is a positive function. It is then easily checked, cf. Section \ref{sambikaehler} below, that the 
first, resp. the second,  condition in (\ref{conf-killing+}) is equivalent to the condition that  the pair 
$(g _+ := f _+ ^{-2} \, g, J _+)$, resp. the pair $(g _- := f _- ^{-2} \, g, J _-)$, is {\it K\"ahler}. On the open set $M _0 = M ^+ _0 \cap M ^- _0$, which is actually dense in $M$, cf. Lemma 
\ref{lemma-M0dense} below, we thus get {\it two} K\"ahler structures, whose metrics belong to the same conformal class and whose complex structures induce opposite orientations (in particular, commute), hence an {\it ambik\"ahler structure}, as defined in \cite{ACG2}. This actually holds if $\psi$ is simply conformal Killing and had been observed in the twistorial setting by M. Pontecorvo in \cite{pontecorvo}, cf. also \cite[Appendix B2]{ACG2}.  The additional coupling condition (\ref{coupling}), which, on $M _0$, reads $J _+  d f _+ = J _- d f _-$, 
cf. Section  \ref{sambikaehler}, has  then
strong additional consequences, that we now explain.   
\smallskip

A first main observation, cf. Proposition \ref{prop-f}, is that the open subset, $M _S$, where $\psi$ is of maximal rank, hence a symplectic $2$-form,  is either empty or dense in $M$. 
\smallskip

The case when $M _S$ is empty is the case when $\psi$ is 
{\it decomposable}, {\it i.e.}  $\psi \wedge \psi = 0$ everywhere; 
equivalently, $|\psi _+| = |\psi _-|$  everywhere;  on $M _0$, we then have $f _+ = f _-$, hence  $g _+ = g _- =: g _K$, and $(M _0, g _K)$ is locally a product of two (real) K\"ahler surfaces $(\Sigma, g _{\Sigma}, \omega _{\Sigma})$ and 
$(\tilde{\Sigma}, g _{\tilde{\Sigma}}, \omega _{\tilde{\Sigma}})$, whereas  $f _+ = f _-$ is constant on $\tilde{\Sigma}$, cf. Section \ref{sdecomposable}. In this case, no non-trivial Killing vector field shows up in general, but a number of compact examples involving Killing vector fields are provided, coming from  \cite{am}. 
\smallskip
 
The case when $M _S$ is dense is first handled with in Proposition \ref{prop-killing}, where we show that the vector 
field $K _1 := - \frac{1}{2} \, 
 \alpha ^{\sharp}$ is then Killing with respect to $g$ --- the chosen normalization is for further convenience --- and that each  eigenvalue of the   Ricci tensor, ${\rm Ric}$,  of $g$ is of multiplicity at least 2; moreover, on the (dense) open set $M _1 = M _S \cap M _0$, $K _1$ 
is Killing with respect to $g _+, g _-$ and Hamiltonian with respect to the K\"ahler forms $\omega _+ := g _+ (J _+ \cdot, \cdot)$ and $\omega _- := 
g _- (J_- \cdot, \cdot)$, whereas ${\rm Ric}$ is both $J _+$- and $J _-$-invariant, cf. Proposition \ref{prop-killing} below. 
On $M _1$, the ambik\"ahler structure $(g _+, J _+, \omega _+)$, $(g _-, J _-, \omega _-)$ is then of the type described in Proposition 11 (iii) 
of \cite{ACG2}. 
\smallskip

In Section \ref{sorthogonal}, we set the stage for a  {\it separation of variables} by introducing new functions $x, y$, defined by $x = \frac{1}{2} (f _+ + f _-)$ and
$y = \frac{1}{2} (f _+ - f _-)$, which, up to a factor $2$, are the ``eigenvalues'' of  $\psi$, and whose gradients are easily shown to be orthogonal. In Proposition \ref{prop-separate}, we show that $|d x| ^2 = A (x)$ and $|d y| ^2 = B (y)$, for some positive functions $A$ and $B$ of one variable. In terms of the new functions $x, y$, the dual $1$-form of $K _1$ with respect to $g$ is simply 
$ J _+ d x + J _+ d y$, whereas in Proposition \ref{prop-killingvf}  a second Killing vector field, $K _2$, shows up, whose dual $1$-form is 
$y ^2 \, J _+ d x + x ^2 \, J _+ dy$ and which turns out to coincide, up to a constant factor, with the Killing vector field  constructed  by W. Jelonek in \cite[Lemma B]{jelonek}, cf. also the proof of Proposition 11 in \cite{ACG2},  namely the image of $K _1$ by the {\it Killing symmetric endomorphism}  $S = \Psi _+ \circ \Psi _- + \frac{(f _+ ^2 + f _- ^2)}{2} {\rm I}$, cf. Remark \ref{rem-jelonek}. 
\smallskip

In Proposition \ref{prop-f}, we then show that either $K _2$ is a (positive) constant  multiple of $K _1$, and we end up with an ambik\"ahler structure {\it of Calabi type}, according to Definition \ref{defi-calabi} taken from \cite{ACG1}, or $K _1, K _2$ are independent on a dense open subset of $M$, determining  an {\it ambitoric structure}, as defined in \cite{ACG2}, \cite{ACG2bis}. 
\smallskip

The Calabi  case is considered in Section \ref{scalabi}, where it is shown that, conversely, any ambik\"ahler structure of Calabi type gives rise, up to scaling, to a $1$-parameter family of pairs $(g ^{(k)}, \psi ^{(k)})$, where  $g ^{(k)}$ is a Riemannian metric in the conformal class and $\psi ^{(k)}$ a $*$-Killing $2$-form with respect to $g ^{(k)}$, cf. Theorem \ref{thm-calabi} and Remark \ref{rem-k}. The example of {\it Hirzebruch-like} ruled surfaces is described in Section \ref{shirzebruch}. 
\smallskip

The ambitoric case is the case when $d x$ and $d y$ are independent on a dense open subset of $M$. In Section \ref{sambitoric}, we show that $x, y$ can be locally completed into a full system of coordinates by the addition of two  ``angular coordinates'', $s, t$, in such a way that $K _1 = \frac{\partial}{\partial s}$ and $K _2 = \frac{\partial}{\partial t}$ and giving rise to a general Ansatz, described in Theorem \ref{thm-ambitoric}. As an Ansatz for the underlying ambik\"ahler structure, this turns out to  be the same as the ambitoric Ansatz of  Proposition  13 in \cite{ACG2} 
for the ``quadratic'' polynomial $q (z) = 2 z$, 
hence  in the {\it hyperbolic} normal form  of \cite[Section 5.4]{ACG2}, when the functions $x, y$  are identified with the {\it adapted coordinates} $x, y$ in \cite{ACG2}.  
\smallskip

The main observation at this point  is that, while the adapted  coordinates  in \cite{ACG2} are obtained via a quadratic transformation, cf. \cite[Section 4.3]{ACG2}, the functions $x, y$ are here naturally attached to the $*$-Killing $2$-form $\psi$ which determines the ambitoric structure.  This is quite reminiscent of the {\it orthotoric} situation, described in \cite{ACG1} in dimension $4$ and in \cite{ACG3} in all dimensions, where the separation of variables --- and the corresponding Ansatz --- are similarly obtained via the ``eigenvalues'' of a {\it Hamiltonian $2$-form}, which share  the same properties as the ``eigenvalues'' $x, y$ of the $*$-Killing $2$-form $\psi$.  
\smallskip

In spite of this, the $*$-Killing $2$-forms considered in this paper are {\it not} Hamiltonian $2$-forms in general --- for a general discussion  about Killing or $*$-Killing $2$-forms versus  Hamiltonian $2$-forms, cf. \cite{au}, in particular Theorem 4.5 and Proposition 4.8,  and, 
also, \cite[Appendix A]{ACG3} ---  but,  
in many respects, at least in dimension $4$, the role  played by Hamiltonian $2$-forms in the orthotoric case is played by $*$-Killing $2$-forms in the (hyperbolic) ambitoric case.
\smallskip

The three situations described above, namely the decomposable, the Calabi ambik\"ahler and the ambitoric case, cf. Proposition \ref{prop-f},   are nicely illustrated  in the example of the round 4-sphere described in Section 
\ref{ssphere}, on which  every $*$-Killing form can be written as the restriction of a constant $2$-form ${\sf a} \in\mathfrak{so}(5)\simeq\Lambda^2\mathbb{R}^5$, which is also the $2$-form associated to the covariant derivative of the Killing vector field induced by ${\sf a}$.  If ${\sf a}$ has rank $2$, the same holds for its restriction on a dense open subset of the sphere, so this corresponds to the decomposable case. Otherwise, ${\sf a}$ can be expressed as $\lambda \, e_1 \wedge e_2 + \mu \, e_3 \wedge e_4$  --- cf. Section \ref{ssphere} for the notation --- with 
$\lambda, \mu$ both positive, and,  depending on whether $\lambda$ and 
$\mu$ are equal or not,  we obtain on a dense subset of the sphere an ambik\"ahler structure of Calabi type or a hyperbolic  ambitoric structure respectively. By using the hyperbolic ambitoric Ansatz of Section 4,  it is eventually shown that the resulting $*$-Killing $2$-forms are actually $*$-Killing with respect to infinitely many non-isometric Riemannian metrics on $S ^4$, cf. Remark \ref{rem-deformation}. 

{\sc Acknowledgments.} We warmly thank Vestislav Apostolov and David Calderbank for their interest in this work and for many useful suggestions. This work was partially supported by the Procope Project No. 32977YJ.

\section{Killing $2$-forms and ambik\"ahler structures} \label{sambikaehler}

In what follows,  $(M, g)$ denotes  a connected,  oriented, $4$-dimensional Riemannian manifold admitting  a non-parallel Killing $2$-form $\varphi$, and $\psi := * \varphi$ denotes the corresponding $*$-Killing $2$-form;  we then have
\begin{equation} \label{nablapsi} \nabla _X \psi = \alpha \wedge X ^{\flat}, \end{equation} 
for some real, non-zero,  $1$-form $\alpha$, where $\nabla$ denotes the Levi-Civita connection of $g$ and $X ^{\flat}$ the dual $1$-form of  $X$ with respect to $g$, cf. \cite{uwe}.  By anti-symmetrizing and by contracting (\ref{nablapsi}), it is easily checked that $\psi$ is closed and that 
\begin{equation}\label{2.2} \delta \psi =3 \alpha, 
\end{equation}
where $\delta$ denotes the codifferential with respect to $g$.   
Denote by $\psi _+ = \frac{1}{2} (\psi + * \psi)$, resp. $\psi _- = \frac{1}{2} (\psi - * \psi)$, the self-dual, resp. the anti-self-dual, part of $\psi$, where $*$ is the Hodge operator induced by the metric $g$ and the chosen orientation. Then, (\ref{nablapsi}) is equivalent to the following two conditions 
\begin{equation} \begin{split} 
& \nabla _X \psi _+ = \big(\alpha \wedge X ^{\flat} \big) _+ = \frac{1}{2} 
\alpha \wedge X ^{\flat} + \frac{1}{2} X \lrcorner* \alpha, \\ & 
\nabla _X \psi _- = \big(\alpha \wedge X ^{\flat} \big) _- = \frac{1}{2} 
\alpha \wedge X ^{\flat} -  \frac{1}{2} X \lrcorner  * \alpha. \end{split}
\end{equation}
Here, we used the general identity:
\begin{equation} \label{identity} * (X ^{\flat} \wedge \phi) = (- 1) ^p \, X \lrcorner  * \phi, \end{equation}
for any vector field $X$ and any $p$-form $\phi$ on any oriented Riemannian manifold. 
In particular, $\psi _+$ and $\psi _-$ are  {\it conformally Killing}, cf. \cite{uwe}. The datum of a (non-parallel) $*$-Killing $2$-form $\psi$ on $(M, g)$ is then equivalent to the datum of a pair $(\psi _+, \psi _-)$ consisting 
of a  self-dual $2$-form $\psi _+$ and an anti-self-dual $2$-form $\psi _-$, both conformally Killing and linked together by
\begin{equation} d \psi _+  + d \psi _- = 0, \end{equation}
or, equivalently, by
\begin{equation} \delta  \psi _+  = \delta  \psi _-. \end{equation}

We denote by $\Psi$, $\Psi _+$, $\Psi _-$ the anti-symmetric endomorphisms of $T M$ associated to $\psi$, $\psi _+$, $\psi _-$ respectively via the metric $g$, so that
$g (\Psi (X), Y) = \psi (X, Y)$, $g (\Psi _+ (X), Y) = \psi _+ (X, Y)$, 
$g (\Psi _- (X), Y) = \psi _- (X, Y)$. On the open set, $M _0$, of $M$ where $\Psi _+$ and $\Psi _-$ have no zero, denote by $J _+$, $J _-$ the corresponding almost complex structures:
\begin{equation} \label{J+-} J _+ := \frac{\Psi _+}{f_+},  \qquad 
J _- := \frac{\Psi _-}{f_-}, \end{equation}
where the positive functions $f_+, f_-$ are defined by
\begin{equation} \label{f+-} f _+ := \frac{|\Psi _+|}{\sqrt{2}}, \qquad f _- := 
\frac{|\Psi _-|}{\sqrt{2}} \end{equation}
(here, the norms $|\Psi _+|$, $|\Psi _-|$, are relative to the conformally invariant inner product defined on the space of anti-symmetric endomorphisms of $TM$  by $(A, B) := - \frac{1}{2} {\rm tr}{(A \circ B)}$); 
the open set $M _0$ is then defined by the condition
\begin{equation} \label{M0} f _+ > 0, \qquad f _- >  0. \end{equation}
Notice that $J _+$ and $J _-$ induce opposite orientations, hence commute to each other, so that the endomorphism
\begin{equation} \label{tau} \tau := - J _+ J _- = - J _- J _+, \end{equation}
is an involution of the tangent bundle of $M _0$.  
\smallskip

From (\ref{nablapsi}), we get 
\begin{equation} \label{nablaA} \nabla _X \Psi = \alpha \wedge X, \end{equation}
with the following general convention: for any $1$-form $\alpha$ and any vector field $X$,  $\alpha \wedge X$ denotes the anti-symmetric endomorphism of $T M$ defined by $(\alpha \wedge X) (Y) = \alpha (Y) X - g (X, Y) \alpha ^{\sharp}$, where $\alpha ^{\sharp}$ is the dual vector field to $\alpha$ relative to $g$  (notice that the latter expression is actually independent of $g$ in the conformal class $[g]$ of $g$). Equivalently:
\begin{equation} \label{nablaA+-} \nabla _X \Psi _+ 
= (\alpha \wedge X) _+, \qquad  
\nabla _X \Psi _- = (\alpha \wedge X) _-. \end{equation}
We infer 
$(\nabla _X \Psi _+, \Psi _+) = \frac{1}{2} (d |\Psi _+| ^2) (X) = (\Psi _+, \alpha \wedge X) = \big(\Psi _+ (\alpha)\big) (X)$, hence $\Psi _+ (\alpha) = \frac{1}{2} d |\Psi _+| ^2$. Similarly, $\Psi _- (\alpha) = \frac{1}{2} d |\Psi _-| ^2$. 
By using (\ref{J+-}), we then get
\begin{equation} \label{alpha-J} \begin{split} & \alpha \,  =  - 2 \Psi _+ \left(\frac{d |\Psi _+|}{|\Psi _+|}\right) = - 2 J _+ d f _+ \\ 
& \quad  = - 2 \Psi _- \left(\frac{d |\Psi _-|}{|\Psi _-|}\right) = - 2 J _- df _-. \end{split} \end{equation}
In particular, 
\begin{equation} \label{Jdf} J _+  d f_+ = J _- d f _-. \end{equation}
\begin{rem} \label{rem-Skilling} {\rm For any $*$-Killing $2$-form $\psi$ as above, denote by $\Phi = \Psi _+ - \Psi _-$ the skew-symmetric endomorphism associated to the Killing $2$-form $\varphi = * \psi$  and by $S$ the symmetric endomorphism defined by
\begin{equation} \label{S} S = - \frac{1}{2} \, \Phi \circ \Phi =   \Psi _+ \circ \Psi _- + \frac{1}{2} (f _+ ^2 + f _- ^2)  \, {\rm I} =  \frac{1}{2} \, \Psi \circ \Psi +  (f _+ ^2 + f _- ^2) \, {\rm I},  \end{equation}
where ${\rm I}$ denotes the identity of $TM$. Then, $S$  is {\it Killing} with respect to $g$, meaning that the symmetric part of $\nabla S$ is zero or, 
equivalently, that $g ((\nabla _X S) X, X) = 0$ for any vector field $X$, cf. 
\cite{pw}, \cite[Appendix B]{ACG2}. This readily follows from the fact that 
$\nabla _X \Phi (X) = X \lrcorner * (\alpha \wedge X) = 0$, so that $g (\nabla _X S (X), X) = - 2 g (\nabla _X \Phi (X), \Phi (X)) = 0$, for any vector field $X$.}
\end{rem}

\begin{lemma} \label{lemma-M0dense}  The open subset $M _0$ defined by {\rm \eqref{M0}} is dense in $M$.
\end{lemma}
\begin{proof} Denote by $M_0 ^{\pm}$ the open set where $f _{\pm} \neq 0$, so that $M _0 = M_0 ^{+} \cap M_0 ^{-}$. It is sufficient to show that each $M ^{\pm}_0$ 
is dense. If not, $f _{\pm} = 0$ on some non-empty open set, $V$,  of $M$, so that  $\psi _{\pm} = 0$ 
on $V$, hence is identically zero, since $\psi _{\pm}$ is conformally Killing, cf. \cite{uwe}; this, in turn, implies that $\alpha$, hence also $\nabla \psi$,  is identically zero, in contradiction to  the hypothesis that $\psi$ is non-parallel.  
\end{proof}
In view of the next proposition, we recall  the following definition, taken from \cite{ACG2}:
\begin{defi}[\cite{ACG2}] \label{defi-ambikaehler} {\rm An {\it ambik\"ahler structure} on an oriented $4$-manifold $M$ consists of a pair of K\"ahler structures, $\big(g_+, J _+, \omega _+ = g_+ (J_+ \cdot, \cdot)\big)$ and $\big(g_-, J_-, \omega _- = g_- (J_- \cdot, \cdot)\big)$, where the Riemannian metrics $g _+, g _-$ belong to the same conformal class, i.e. $g _- = f ^2 \, g _+$, for some positive function $f$,  and the complex structure $J _+$, resp. the complex structure $J _-$,  induces the chosen orientation, resp. the opposite orientation; equivalently, the K\"ahler forms $\omega _+$ and $\omega _-$ are self-dual and  anti-self-dual respectively. }
\end{defi}
We then have:
\begin{prop} \label{prop-ambikaehler} Let $(M, g)$ be a connected,  oriented, 
 $4$-dimensional Riemannian manifold, equipped with a non-parallel  $*$-Killing $2$-form $\psi = \psi _+ + \psi _-$ as above.  Then, on the dense open subset, $M _0$,  of $M$ defined by {\rm (\ref{M0})}, the pair $(g, \psi)$ gives rise to an ambik\"ahler structure  
$(g_+, J_+, \omega _+)$, $(g_-, J_-, \omega_-)$, with $g _{\pm} = f _{\pm} ^{-2} \, g$ and  $J _{\pm} = f _{\pm} ^{-1} \, \Psi _{\pm}$, by setting $f _{\pm} = |\Psi _{\pm}|/\sqrt{2}$. In particular, this ambik\"ahler structure is equipped with two non-constant positive functions $f_+, f_-$, satisfying the two conditions 
\begin{equation} \label{f} f = \frac{f_+}{f_-}, \end{equation}
and
\begin{equation} \label{taudf} \tau (df_+) = d f _-. \end{equation}

Conversely, any ambik\"ahler structure $(g_+, J_+, \omega _+)$, $(g_- = f ^2 \, g _+, J_-, \omega_-)$ equipped with two non-constant positive functions $f _+, f _-$ satisfying {\rm (\ref{f})--(\ref{taudf})} arises from a unique pair $(g, \psi)$, where $g$ is the Riemannian metric in the conformal class $[g_+] = [g _-]$ defined by
\begin{equation} \label{g} g = f _+ ^2 \, g _+ = f _- ^2 \, g _-, \end{equation}
and $\psi$ is the $*$-Killing $2$-form relative to $g$ defined by 
\begin{equation} \label{psi} \psi = f _+ ^3 \, \omega _+ + f _- ^3 \, \omega _-. \end{equation} 
\end{prop} 
\begin{proof} Before starting the proof, we recall the following general facts. (i) For any two Riemannian metrics, $g$ and $\tilde{g} = \varphi ^{-2} \, g$, in a same conformal class, and for any anti-symmetric endomorphism, $A$, of the tangent bundle with respect to the conformal class $[g] = [\tilde{g}]$,  the covariant derivatives $\nabla ^{\tilde{g}} A$ and $\nabla ^{g} A$ are related by
\begin{equation} \label{nabla-conf} \nabla ^{\tilde{g}} _X A = \nabla ^g _X A + \left[A, \frac{d \varphi}{\varphi} \wedge X\right] = A \left(\frac{d \varphi}{\varphi}\right) \wedge X + \frac{d \varphi}{\varphi} \wedge A (X),  \end{equation} 
by setting  $A \left(\frac{d \varphi}{\varphi}\right) = - \frac{d \varphi}{\varphi} \circ A$. 
(ii) For any $1$-form $\beta$ and any vector field $X$, we have
\begin{equation} \label{betaX+} \begin{split} (\beta \wedge X) _+ & = \frac{1}{2} \beta \wedge X - \frac{1}{2} J _+ \beta \wedge J _+ X -  \frac{1}{2} \beta (J _+ X) \, J _+ \\ & = \frac{1}{2} \beta \wedge X +  \frac{1}{2} J _- \beta \wedge J _- X + \frac{1}{2} \beta (J _-  X) \, J _-, \end{split} \end{equation}
and
\begin{equation} \label{betaX-} \begin{split} (\beta \wedge X) _- & = \frac{1}{2} \beta \wedge X - \frac{1}{2} J _- \beta \wedge J _- X -  \frac{1}{2} 
\beta (J _- X) \, J _- \\ & = \frac{1}{2} \beta \wedge X +  \frac{1}{2} J _+ \beta \wedge J _+ X + \frac{1}{2} \beta (J _+ X) \, J _+, \end{split} \end{equation}
for {\it any} orthogonal (almost) complex structures $J _+$ and $J -$ inducing the chosen and the opposite orientation respectively. 
\smallskip

From (\ref{J+-}), (\ref{nablaA+-}), (\ref{alpha-J}) and (\ref{betaX+}), we thus infer
\begin{equation} \label{nablaJ+} \begin{split} \nabla _X J _+ & 
= - 2 \left(J _+ \left(\frac{d f_+}{|f_+|}\right) \wedge X\right) _+ 
- \frac{d f_+}{f_+} (X) \, J _+ \\ & = 
- J _+ \left(\frac{d f_+}{f_+}\right) \wedge X - \frac{d f_+}{f_+} \wedge J_+ X + \frac{d f_+}{f_+} (X) \, J _+ - \frac{d f_+}{f_+} (X) \, J _+ 
\\ & = - J _+ \left(\frac{d f_+}{f_+}\right) \wedge X - \frac{d f_+}{f_+} \wedge J_+ X
= \left[\frac{d f_+}{f_+} \wedge X, J_+\right] \end{split} \end{equation}
which, by using (\ref{nabla-conf}), is equivalent to
\begin{equation}\label{nJ+} \nabla ^{g_+} J _+ = 0, \end{equation}
where $\nabla ^{g_+}$ denotes the Levi-Civita connection of the conformal metric 
$g _+ = f_+ ^{- 2} \, g$, meaning that the pair $(g_+, J_+)$ is {\it K\"ahler}.  
Similarly, we have
\begin{equation} \label{nablaJ-} \nabla _X J _- = \left[\frac{d f_-}{f_-} \wedge X, J _-\right] \end{equation}
or, equivalently:
\begin{equation}\label{nJ-} \nabla ^{g_-} J _- = 0, \end{equation}
where $\nabla ^{g_-}$ denotes the Levi-Civita connection of the conformal metric 
$g _- = f_- ^{- 2} \, g$, meaning that the pair $(g_-, J_-)$ is K\"ahler as well.
We thus get on $M _0$ an {\it ambik\"ahler structure} in the sense of Definition \ref{defi-ambikaehler}.
Moreover, because of (\ref{Jdf}), $f_+$ and $f_-$ evidently satisfy (\ref{f})--(\ref{taudf}).

For the converse, define $g$ by 
\begin{equation} g = f _+ ^2 \, g_+ = f _- ^2 \, g_- \end{equation}
and denote by $\nabla$ the Levi-Civita connection of $g$. 
 By defining $\Psi _+ = f _+ \, J_+$, $\Psi _- = f_-  \, J_-$ and $\Psi = \Psi _+ + \Psi _-$,  
we get
\begin{equation} \begin{split} \nabla _X \Psi _+ & =  \nabla _X (f _+ \, J _+) \\ & =  \nabla ^{g_+} _X (f _+ \, J_+) + \left[\frac{d f_+}{f_+} \wedge X, f _+ \, J_+\right] \\ & =  d f _+ (X) \, J _+ - J _+  d f _+ \wedge X - d f _+ \wedge J_+ X \\ & = - 2  (J _+ d f _+ \wedge X)_+. \end{split} \end{equation}
Similarly, 
\begin{equation} \nabla _X \Psi _- = - 2 \,  (J _- d f_- \wedge X) _-. \end{equation}
By using (\ref{Jdf}), we obtain
\begin{equation} \nabla _X \Psi = \alpha \wedge X, \end{equation}
with 
$\alpha: = - 2  \,  J _+ d f _+ = - 2 \,  J _- d f _-$,
 meaning that the associated $2$-form $\psi (X, Y): = g (\Psi (X), Y)$, is $*$-Killing. Finally $\psi = f _+ \, g (J _+ \cdot, \cdot) + f _- \, g (J _- \cdot, \cdot) = f _+ ^3 \, \omega _+ + f _- ^3 \, \omega _-$. 
\end{proof} 
\begin{rem} {\rm The fact that the pair  $(g _+ = f _+ ^{-2} \, g, J _+)$, resp. the pair $(g _- = f _- ^{-2} \, g, J _-)$,  is K\"ahler only depends on, in fact is equivalent to,  $\Psi _+ = f _+ \, J _+$, resp. $\Psi _- = f _- \, J _-$,  being conformal Killing, i.e. $\psi$ being conformally Killing. This was observed in a twistorial setting by M. Pontecorvo in \cite{pontecorvo}, cf. also Appendix B2 in \cite{ACG2}.
}
\end{rem}
We now explain under which circumstances an ambik\"ahler structure satisfies the conditions (\ref{f})--(\ref{taudf}). 
\begin{prop} \label{prop-kappa} Let $M$ be an oriented $4$-manifold equipped with an ambik\"ahler structure $(g_+, J_+, \omega_+)$, $(g_- = f ^2 \, g _+, J_-, \omega_-)$. Assume moreover that $f$ is not constant. Then, on the open set 
where $f \neq 1$, there exist non-constant positive functions $f _+, f _-$ satisfying {\rm (\ref{f})--(\ref{taudf})} of {\rm Proposition \ref{prop-ambikaehler}}  if and only if the $1$-form
\begin{equation} \label{kappa} \kappa := \frac{\tau (df)}{1 - f ^2} \end{equation}
is exact.
\end{prop}
\begin{proof} 
For any ambik\"ahler structure $(g_+, J_+, \omega_+)$, $(g_- = f^2 \, g_+, J_-, \omega_-)$ and any positive functions $f _+$, $f _-$ satisfying  (\ref{f})--(\ref{taudf}), we have 
\begin{equation} \label{1-f} \begin{split} & (1 - f ^2) \, \frac{df _+}{f_+} = 
\frac{df}{f} + \tau (df), \\ 
& (1 - f ^2) \, \frac{df _-}{f_-} = f \, df + \tau (df). \end{split} \end{equation}
On the open set where $f \neq 1$, this can be rewritten as 
\begin{equation} \label{taudf+-} \begin{split} & 
\frac{d f_+}{f_+} = \frac{d f}{f (1 - f ^2)} + \frac{\tau (df)}{(1 - f ^2)}, \\ & 
\frac{d f_-}{f_-} = \frac{f d f}{(1 - f ^2)} + \frac{\tau (df)}{(1 - f ^2)}; \end{split} \end{equation}
in particular, $\kappa$ is exact on this open set. 
 Conversely, if $\kappa$ is exact, but not identically zero, then $\kappa = \frac{d \varphi}{\varphi}$, for some, non-constant,  positive function,  $\varphi$,  and we then {\it define} $f _+, f _-$ by 
$\frac{df _+}{f_+} = \frac{d \varphi}{\varphi} + \frac{d f}{f (1 - f ^2)}$ and  $\frac{d f _-}{f_-} = \frac{d \varphi}{\varphi} + \frac{f \, d f}{(1 - f ^2)}$, hence by 
$f _+ := \frac{f \, \varphi}{|1 - f ^2| ^{\frac{1}{2}}}$ and 
$f _- := \frac{\varphi}{|1 -  f ^2| ^{\frac{1}{2}}}$, which clearly satisfy (\ref{f})--(\ref{taudf}).
\end{proof}
\begin{rem} \label{rem-constant} {\rm It follows from (\ref{1-f}) that if $f = k$, where $k$ is a constant different from $1$, then $f _+$ and $f _-$ are constant and the corresponding $*$-Killing $2$-form $\psi$ is then parallel.
More generally, the existence of a pair $(g, \psi)$ inducing an ambik\"ahler structure depends on the chosen relative scaling of the K\"ahler metrics. More precisely, if the ambik\"ahler structure $(g_+, J_+, \omega_+)$, $(g_- = f ^2 \, g _+, J _-, \omega _-)$ arises from a $*$-Killing $2$-form in the conformal class, in the sense of Proposition \ref{prop-ambikaehler}, then for any positive constant $k \neq 1$, the ambik\"ahler structure 
$(g_+, J_+, \omega_+)$, $(\tilde{g} _- = k ^2 \, g _-, J _-, k^2\omega _-)$ {\it does not} arise from a $*$-Killing $2$-form, unless $\tau (df) = \pm df$. This is because the $1$-forms $\frac{\tau (df)}{(1 - f ^2)}$ and $\frac{\tau (df)}{(1 - k ^2 \, f ^2)}$ would then be both closed, implying that $\tau (df) = \phi \, df$ for some function $\phi$; since $|\tau (df)| = |d f|$, we would then 
have $\phi = \pm 1$. 
} \end{rem}
The $1$-form $\kappa$ in Proposition \ref{prop-kappa} is clearly exact on the open set where $f \neq 1$ whenever $\tau (df) = df$ or $\tau (df) = - df$, and it readily follows from (\ref{taudf+-}) that $f _+, f_-$ are then given by
\begin{equation} \label{f+} f _+ = \frac{c \, f}{|1 - f|}, \quad f _- = \frac{c}{|1 - f|} = \pm c + f _+, \end{equation}
if $\tau (df) = df$, or by
\begin{equation} \label{f-} f _+ = \frac{c \, f}{1 + f}, \quad f _- = \frac{c}{1 + f} = c - f _+, \end{equation}
if $\tau (df) = - df$, for some positive constant $c$.   If 
\begin{equation} \label{split} T M _0 = T^+ \oplus T ^-, \end{equation}
denotes the orthogonal splitting determined by $\tau$, where $\tau$ is the identity on $T ^+$ and minus the identity on $T ^-$ --- equivalently, $J _+$, $J _-$ coincide on $T ^+$ and are opposite on $T ^-$ --- then $\tau (df) = \pm
df$ if and only if $df _{|T ^{\mp}} = 0$ and we also have:
\begin{prop} \label{prop-involutive}
The distribution $T ^{\pm}$  is involutive if and only if $\tau (df) = \pm df$.
\end{prop}
\begin{proof}
For a general ambik\"ahler structure 
$(g _+, J_+, \omega_+)$ and 
$(g _- = f ^2 \, g_+, J _-, \omega_-)$, with $g _- = f ^2 \, g _+$, we have
\begin{equation} \label{dfZ-} 
\frac{df (Z)}{f} \, \omega _+ (X, Y) = - \omega _+ ([X, Y], Z), \quad \frac{df (Z)}{f} \, \omega _- (X, Y) =  
\omega _- ([X, Y], Z), \end{equation}
for any $X, Y$ in $T ^+$ and any $Z$ in $T ^-$, and
\begin{equation} \label{dfZ+} 
\frac{df (Z)}{f} \, \omega _+ (X, Y) =  \omega _+ ([X, Y], Z), \quad  
\frac{df (Z)}{f} \, \omega _- (X, Y) =  
- \omega _- ([X, Y], Z),  \end{equation}
for any $X, Y$ in $T ^-$ and any $Z$ in $T ^+$. This can be shown as follows. 
Suppose that $X, Y$ are in $T ^+$ and $Z$ is in $T ^-$. Then, since the K\"ahler form $\omega _+ (\cdot, \cdot) = g _+ (J _+ \cdot, \cdot)$ and $\omega _- (\cdot, \cdot) = g (J_- \cdot, \cdot)$ are closed and $T ^+, T^-$ are $\omega _+$- and $\omega _-$-orthogonal, we have
\begin{equation} \label{domega+} 
Z \cdot \omega _+ (X, Y) = \omega _+ ([X, Y], Z) 
+  \omega _+ ([Y, Z], X) +  \omega _+ ([Z, X], Y), \end{equation}
and
\begin{equation} \label{domega-} Z \cdot \omega _- (X, Y) = \omega _- ([X, Y], Z) 
+  \omega _- ([Y, Z], X) +  \omega _- ([Z, X], Y), \end{equation}
which can be rewritten as
\begin{equation} \label{domega-+} 
Z \cdot \left(f ^2 \omega _+ (X, Y)\right) = - f ^2 \, 
\omega _+   ([X, Y], Z) 
+  f ^2 \, \omega _+ ([Y, Z], X) +  f ^2 \, \omega _+ ([Z, X], Y),  
\end{equation}
or else:
\begin{equation} \label{domega-+-bis} \begin{split} 
2 \frac{df (Z)}{f} \, \omega _+ (X, Y) & + Z \cdot \omega _+ (X, Y)  = \\ &  -  
\omega _+   ([X, Y], Z) 
+  \omega _+ ([Y, Z], X) +   \omega _+ ([Z, X], Y).  
\end{split} \end{equation}
Comparing (\ref{domega+}) and (\ref{domega-+-bis}), we readily deduce the first identity in (\ref{dfZ-}); the other three identities are checked similarly. Proposition \ref{prop-involutive} then readily follows from (\ref{dfZ-})--(\ref{dfZ+}).
\end{proof}

In the following statement, $M _0$ stills denotes the (dense) open subset of $M$ defined by (\ref{M0}); we also denote by $M _S$ the open subset of $M$ defined by
\begin{equation} \label{MS} f _+ \neq f _-, \end{equation}
on which $\psi$ is a symplectic $2$-form, and by $M _1$ the intersection 
$M _1 := M _0 \cap M _S$. 
\begin{prop} \label{prop-killing}  Let $(M, g)$ be an oriented Riemannian $4$-dimensional manifold 
admitting a non-parallel  $*$-Killing $2$-form $\psi$. Denote by 
$(g _+ = f _+ ^2 \, g, J _+, \omega _+)$, 
$(g _- = f _- ^2 \, g, J _-, \omega _-)$ the induced ambik\"ahler structure on $M _0$ as explained above. Then, on the open set $M _1$, the Ricci endomorphism, ${\rm Ric}$,  of $g$ is $J_+$- and $J_-$-invariant, hence of the form
\begin{equation} \label{Ric} {\rm Ric} = a \, {\rm I} + b \, \tau, 
\end{equation}
for some functions $a, b$, where ${\rm I}$ denotes the identity of $TM _1$ and $\tau$ is defined by {\rm (\ref{tau})}. Moreover,   the vector field 
\begin{equation} \label{K} K_1 := J _+ {\rm grad} _g f_+ = 
J _- {\rm grad} _g f _- = - \frac{1}{2} \alpha ^{\sharp} \end{equation}
is Killing with respect to $g$ and preserves the whole ambik\"ahler structure.
\end{prop}
\begin{proof}
Let ${\rm R}$ be the curvature tensor of $g$, defined by
\begin{equation} {\rm R} _{X, Y} Z: = \nabla _{[X, Y]} Z - [\nabla _X, \nabla _Y] Z, \end{equation} 
for any vector field $X, Y, Z$. We denote by ${\rm Scal}$ its scalar curvature, 
by ${\rm Ric}  _0$ the trace-free part of ${\rm Ric}$, by ${\rm W}$ the Weyl tensor of $g$, and by ${\rm W} ^+$ and ${\rm W} ^-$ its self-dual and anti-self-dual part respectively.
As in the previous section, $\Psi$ denotes the skew-symmetric endomorphism of $TM$ determined by $\psi$, $\Psi _+$ its self-dual part, $\Psi _-$ its anti-self-dual part, with 
$\Psi _+ = f _+  \,  J_+$ and $\Psi _- = f _-  \,  J _-$ on $M _0$. Since $g = f _+ ^2 \, g _+ = f _- ^2 \, g _-$, where $g _+$ and $g _-$ are K\"ahler with respect to $J_+$ and $J _-$ respectively, ${\rm W} ^+$ and ${\rm W} ^-$ are both degenerate 
and ${\rm W} ^+ (\Psi _+) = 
\lambda _+ \,  \Psi _+$, ${\rm W} ^- (\Psi _-) = \lambda _- \Psi _-$, for some functions $\lambda _+, \lambda _-$. 
For any vector fields $X, Y$ on $M$, the usual decomposition of the curvature tensor reads:
\begin{equation} \label{RA} \begin{split} 
{\rm R} _{X, Y} \Psi & = [{\rm R} (X  \wedge Y), \Psi] \\ & = 
\frac{{\rm Scal}}{12} \, [X^{\flat}  \wedge Y, \Psi] + \frac{1}{2} [\{{\rm Ric} _0, X ^{\flat} \wedge Y\}, \Psi]\\ & + [{\rm W} ^+ (X \wedge Y), \Psi _+] + [{\rm W} ^- (X \wedge Y), \Psi _-], \end{split} \end{equation}
by setting $\{{\rm Ric}_0, X ^{\flat}  \wedge Y\} := {\rm Ric}  _0 \circ (X ^{\flat}  \wedge Y) 
+ (X ^{\flat}  \wedge Y) \circ {\rm Ric}  _0= {\rm Ric}  _0 (X) \wedge Y + X \wedge {\rm Ric}  _0 (Y)$, cf. {\it e.g.} \cite[Chapter 1, Section G]{besse}. 
On $M_0$ we then have:
\begin{equation} \frac{{\rm Scal}}{12} [X \wedge Y, \Psi] = - \frac{{\rm Scal}}{12} \, \big(\Psi (X) \wedge Y + X \wedge \Psi (Y)\big), \end{equation}
\begin{equation} \begin{split} 
\frac{1}{2} [\{{\rm Ric}_0, X \wedge Y\}, \Psi] & =   - \frac{1}{2} \, 
\Big(\Psi \big({\rm Ric} _0 (X)\big) \wedge Y + {\rm Ric} _0 (X) \wedge \Psi (Y) \\ & \qquad \quad + \Psi (X) \wedge {\rm Ric} _0 (Y)  + X \wedge \Psi \big({\rm Ric} _0 (Y)\big)\Big), \end{split} \end{equation}
and
\begin{equation}  \begin{split} & {\rm W} ^+ _{X, Y} \Psi _+ = \frac{\lambda _+}{2} \, \big(\Psi _+ (X) \wedge Y + X \wedge \Psi _+ (Y)\big), \\ & {\rm W} ^- _{X, Y} \Psi _- = \frac{\lambda _-}{2} \big(\Psi _- (X) \wedge Y + X \wedge \Psi _- (Y)\big).  \end{split} \end{equation}
We thus get
\begin{equation} \begin{split} \sum _{i = 1} ^4 e _i \lrcorner {\rm R} _{e _i, Y} \Psi  
 = & \left(\lambda _+ - \frac{{\rm Scal}}{6} \right) \, \Psi _+ (Y) + 
\left(\lambda _- - \frac{{\rm Scal}}{6} \right) \, \Psi _- (Y)\\ & + \frac{1}{2} [{\rm Ric} _0, \Psi] (Y). \end{split} \end{equation}
Similarly,
\begin{equation}  \sum _{i = 1} ^4 e _i \lrcorner {\rm R} _{e _i, Y} \Psi _+  
 =  \left(\lambda _+ - \frac{{\rm Scal}}{6} \right) \, \Psi _+ (Y) 
+ \frac{1}{2} [{\rm Ric} _0, \Psi _+] (Y)   \end{equation}
and
\begin{equation}  \sum _{i = 1} ^4 e _i \lrcorner {\rm R} _{e _i, Y} \Psi _-
 =  \left(\lambda _- - \frac{{\rm Scal}}{6} \right) \, \Psi _- (Y) + \frac{1}{2} [{\rm Ric} _0, \Psi _-] (Y).  \end{equation}
On the other hand, from (\ref{nablaA}), we get
\begin{equation} {\rm R} _{X, Y} \Psi = \nabla _Y \alpha \wedge X - \nabla _X \alpha \wedge Y, \end{equation}
hence
\begin{equation} \sum _{i = 1} ^4 e _i \lrcorner {\rm R} _{e _i, Y} \Psi = - 2 \nabla _Y \alpha, \end{equation}
whereas, from (\ref{nablaA+-}), we obtain
\begin{equation} {\rm R} _{X, Y} \Psi _+ = (\nabla _Y \alpha \wedge X - \nabla _X \alpha \wedge Y) _+, \quad {\rm R} _{X, Y} \Psi _- = (\nabla _Y \alpha \wedge X - \nabla _X \alpha \wedge Y) _-, \end{equation}
hence 
\begin{equation} \sum _{i = 1} ^4 e _i \lrcorner {\rm R} _{e _i, Y} \Psi _+ =  - Y \lrcorner \left(\nabla  \alpha\right) ^s - Y \lrcorner (d \alpha) _+, \end{equation}
where $\left(\nabla  \alpha\right) ^s$ denotes the symmetric part of $\nabla \alpha$. Indeed,  we have 
\begin{equation} \begin{split} \sum _{i = 1} ^4 e _i \lrcorner \big(\nabla _Y \alpha \wedge e_i  - \nabla _{e_i}& \alpha \wedge Y)_+  = \frac{1}{2} 
\sum _{i = 1} ^4 e _i \lrcorner (\nabla _Y \alpha \wedge e _i) - \frac{1}{2} 
\sum _{i = 1} ^4 e _i \lrcorner (\nabla _{e _i} \alpha \wedge Y) 
\\ & + \frac{1}{2} 
\sum _{i = 1} ^4 e _i \lrcorner * (\nabla _Y \alpha \wedge e _i) - 
\frac{1}{2} \sum _{i = 1} ^4 e _i \lrcorner * (\nabla _{e _i} \alpha \wedge Y) \\ &  = - \nabla _Y \alpha - \frac{1}{2} \sum _{i = 1} ^4 e _i \lrcorner * (\nabla _{e _i} \alpha \wedge Y) \\ & = - \nabla _Y \alpha - \frac{1}{2} Y \lrcorner * d \alpha = - Y \lrcorner (\nabla \alpha) ^s - Y \lrcorner (d \alpha) _+, 
\end{split} \end{equation}
as $\delta \alpha = 0$ and $e _i \lrcorner * (\nabla _Y \alpha \wedge e _i)$ is clearly equal to zero thanks to the general identity (\ref{identity}). 
We obtain similarly:
\begin{equation} \sum _{i = 1} ^4 e _i \lrcorner {\rm R} _{e _i, Y} \Psi _- = -  Y \lrcorner  \left(\nabla \alpha\right) ^s - Y \lrcorner (d \alpha) _-. \end{equation}
From the above, we  infer
\begin{equation} \label{nablaalpha} \begin{split} &(d \alpha) _+ = \left(\frac{{\rm Scal}}{6} - \lambda _+\right) \,  \psi _+, \quad 
(d \alpha) _-=  \left(\frac{{\rm Scal}}{6} - 
\lambda _- \right) \, \psi _-, \\ &
\left(\nabla  \alpha\right) ^s = - \frac{1}{2} [{\rm Ric} _0, \Psi _+] = - 
\frac{1}{2} [{\rm Ric} _0, \Psi _-]. \end{split} \end{equation}
It follows that 
\begin{equation} \label{crux}
[{\rm Ric}, \Psi _+] = [{\rm Ric}, \Psi _-], \end{equation}
and that the vector field $\alpha ^{\sharp _g}$ is Killing with respect to $g$ if and only if $[{\rm Ric}, \Psi _+] = [{\rm Ric}, \Psi _-] = 0$.
We now show that (\ref{crux}) actually {\it implies}  $[{\rm Ric}, \Psi _+] = [{\rm Ric}, \Psi _-] = 0$ at each point where $f _+ \neq f _-$.  Indeed, in terms of the decomposition (\ref{split-lambda}), ${\rm Ric}$, $J _+$, $J _-$ can be written in the following matricial form
\begin{equation} {\rm Ric} = \begin{pmatrix} P & Q \\ Q^* & R \end{pmatrix}, \quad J _+ = \begin{pmatrix} J & 0 \\ 0 & J \end{pmatrix}, \quad J _- = 
\begin{pmatrix} J & 0 \\ 0 & - J \end{pmatrix} \end{equation}
where $J$ denotes the restriction of $J _+$ on $T ^+$ and on $T ^-$, so that:
\begin{equation} [{\rm Ric}_0, J_+] = \begin{pmatrix} [P, J] & [Q, J] \\ [Q^*, J] & [R, J] \end{pmatrix}, \quad [{\rm Ric} _0, J_-] = \begin{pmatrix} 
[P, J] & - \{Q, J\} \\ \{Q^*, J\} & - [R, J]. \end{pmatrix} \end{equation} 
Then (\ref{crux}) can be expanded as
\begin{equation} \label{crux-expand} \begin{split} & (f _+ - f _-) [P, J] = 0, 
 \\ & (f _+ + f _-) \, Q J = (f_+ - f _-) \, J Q, \\ & (f _+ + f _-) [R, J] = 0. \end{split} \end{equation}
Since $f_+>0$ and $f_-> 0$ on $M_0$, from (\ref{crux-expand}) we readily infer $[R, J] = 0$ and $Q = 0$, meaning that
\begin{equation} {\rm Ric} = \begin{pmatrix} P & 0 \\ 0 & R \end{pmatrix}.
\end{equation}
Moreover, on the open subset $M _1 = M _0 \cap M _S$, where $f _+ - f _- \neq 0$, we also  infer from (\ref{crux-expand}) that $[P, J] = 0$, hence that $[{\rm Ric}, J _+] = [{\rm Ric}, J_-] = 0$. By (\ref{nablaalpha}), $(\nabla \alpha) ^s = 0$, meaning that the the vector field $K_1 := - \frac{1}{2} \alpha ^{\sharp} = J _+ {\rm grad} _g f_+$ is Killing with respect to $g$.  Notice that
\begin{equation} \begin{split} \label{K+-} K_1&  = J _+ {\rm grad} _g f _+ = J _- {\rm grad} _g f _- \\ & = - J _+ {\rm grad} _{g _+} \frac{1}{f _+} = - J _- {\rm grad} _{g _-} \frac{1}{f_-}. \end{split} \end{equation}
In particular, $K_1$ is also Killing with respect to $g _+$ and $g _-$ and is (real) holomorphic with respect to $J _+$ and $J _-$. 
\end{proof}

\section{Separation of variables} \label{sorthogonal} 

In this section we restrict our attention to the open subset  $M _1 := M _0 \cap M _S$,  defined by the conditions (\ref{M0}) and (\ref{MS}). 
Recall that since $\psi \wedge \psi = \psi _+ \wedge \psi _+ + \psi _- \wedge \psi _- = 
2 (f _+ - f _-) \, v _g$, where $v _g$ denotes the volume form of $g$ relative to the chosen orientation, $M _S$ is the open subset of $M$ where $\psi$ is non-degenerate, hence a symplectic $2$-form. 
According to  Proposition \ref{prop-killing},  on $M _1$ the Ricci tensor ${\rm Ric}$ is of the form (\ref{Ric}), for some functions $a, b$
and the vector field $\alpha ^{\sharp}$ is Killing;  we then infer from (\ref{nablaalpha}) that $\nabla \alpha ^{\sharp}$ can be written as:
\begin{equation} \label{nablaalphasharp} \nabla \alpha ^{\sharp} = h _+ \, J _+ + h _- \, J _-, \end{equation}
with
\begin{equation} \label{h} h _+ := \frac{1}{2} f _+ \left(\frac{\rm Scal}{6} - \lambda _+\right), \quad h _- := \frac{1}{2} f _ - \left(\frac{\rm Scal}{6} - \lambda _-\right). \end{equation}
We then introduce the functions $x, y$ defined by
\begin{equation} \label{xy} \begin{split} 
& x: = \frac{f _+ + f _-}{2}, \quad y: = \frac{f _+ - f _-}{2}, \\ & 
f _+ = x + y, \qquad \quad f _- = x - y.   \end{split} \end{equation}
Notice that $(2 x, 2 y)$, resp. $(2 x, - 2 y)$,  are the  eigenvalues of the Hermitian operator $- J _+ \circ \Psi = f _+ \, {\rm I} + f _- \, \tau$, resp. $- J _- \circ  \Psi = f _+ \,\tau + f _- \, {\rm I}$,  relative to the eigen-subbundle $T ^+$ and $T^-$ respectively. From (\ref{M0}) and (\ref{MS})  we deduce that $x, y$ are subject to the conditions
\begin{equation} \label{condxy} x >  |y| > 0,  \end{equation}
whereas, from (\ref{Jdf}), we infer
\begin{equation} \label{tauxy} \tau (dx) = d x, \qquad \tau (dy) = - dy. \end{equation} 
In particular, $dx$, $J _+ dx = J _- dx$, $dy$ and $J _+ dy = - J _- dy$ are 
{\it pairwise orthogonal} and 
\begin{equation} |dx| ^2 + |dy| ^2 = |d f_+|^2 = |df _-| ^2, \quad 
|dx| ^2 -  |dy| ^2 = (d f _+, d f _-). \end{equation}
We then have:
\begin{prop} \label{prop-separate}   On each connected component of the open subset of $M_1$ where $dx\ne 0$ and $dy\ne0$, the square norm of $dx, dy$ and the Laplacians of $x, y$ relative to $g$ are given by
\begin{equation} \label{separate} \begin{split} &  |d x| ^2 = 
\frac{A (x)}{(x ^2 - y ^2)}, 
\quad |d y| ^2 = \frac{B (y)}{(x ^2 - y ^2)}, \\ & \Delta x = - 
\frac{A ' (x)}{(x ^2 - y ^2)}, \quad \Delta y = - \frac{B ' (y)}{(x ^2 - y ^2)}, \end{split} \end{equation}
where $A, B$ are functions of one variable. 
\end{prop}
\begin{proof} By using (\ref{nablaJ+}) and (\ref{nablaJ-}) and setting $g _{\tau} (X, Y) := g (\tau (X), Y)$, we infer from (\ref{alpha-J}) and (\ref{nablaalphasharp}) that 
\begin{equation} \label{nabladf+-} \begin{split} &
\nabla d f _+ = \left(- \frac{1}{2} h _+ + \frac{|d f _+|^2}{f_+}\right) \, g - \frac{1}{2} \, h _- \, g _{\tau} \\ & \qquad \qquad - \frac{1}{f _+} \big(d f _+ \otimes d f _+  + J _+ d f _+ \otimes J _+ d f _+\big), \\ &
\nabla d f _- = \left(- \frac{1}{2} h _- + \frac{|d f _-|^2}{f_-}\right) \, g - \frac{1}{2} \, h _+ \, g _{\tau} \\ & \qquad \qquad - \frac{1}{f _-} \big(d f _- \otimes d f _-  + J _- 
d f _- \otimes J _- d f _-\big). \end{split} \end{equation}
In terms of the functions $x, y$, this can be rewritten as
\begin{equation} \label{nabladxy} \begin{split} 
& \nabla d x = \big(\frac{x}{(x ^2 - y ^2)} (|d x| ^2 + |d y| ^2) - \frac{1}{4} (h _+ + h _-)  \big) \, g - \frac{1}{4} (h _+ + h _-) \, g _{\tau} \\ & \quad \qquad - \frac{x}{(x ^2 - y ^2)} \, (d x \otimes d x + d y \otimes d y) + 
\frac{y}{(x ^2 - y ^2)} \, (d x \otimes d y + d y \otimes d x) \\ & \quad \qquad
- \frac{x}{(x ^2 - y ^2)} \, J _+ (d x + d y) \otimes J _+ (d x + d y), \\
& \nabla d y = - \big(\frac{y}{(x ^2 - y ^2)} (|d x| ^2 + |d y| ^2) + \frac{1}{4} (h _+ - h _-)  \big) \, g +  \frac{1}{4} (h _+ - h _-) \, g _{\tau} \\ & \quad \qquad + \frac{y}{(x ^2 - y ^2)} \, (d x \otimes d x + d y \otimes d y) - 
\frac{x}{(x ^2 - y ^2)} \, (d x \otimes d y + d y \otimes d x) \\ & \quad \qquad
+  \frac{y}{(x ^2 - y ^2)} \, \big(J _+ (d x + d y) \otimes J _+ (d x + d y)\big).
\end{split} \end{equation}
In particular:
\begin{equation} \label{Deltaxy} \begin{split} & \Delta x = (h _+ + h _-) - \frac{2 x}{(x ^2 - y ^2)} \, (|d x| ^2 + |d y| ^2), \\ & \Delta y = (h _+ - h _-) + \frac{2 y}{(x ^2 - y ^2)} \, (|d x| ^2 + |dy|^2). \end{split} \end{equation}
To simplify the notation, we temporarily put
\begin{equation} F := |d x| ^2, \qquad G := |d y| ^2. \end{equation}
By contracting $\nabla d x$ by $dx$ and $\nabla d y$ by $d y$ in 
(\ref{nabladxy}), and taking (\ref{Deltaxy}) into account, we obtain:
\begin{equation} \begin{split} \label{system}
& d F = - \left(\Delta x  + \frac{2 x \, F}{(x ^2 - y ^2)}\right) \, d x + \frac{2 y \, F}{(x ^2 - y ^2)} \, d y, \\ &
d G = - \frac{2 x \, G}{(x ^2 - y ^2)} \, d x - \left(\Delta y -  
\frac{2 y \, G}{(x ^2 - y ^2)}\right) \, dy.  \end{split} \end{equation}
From (\ref{system}), we get
\begin{equation} \begin{split} 
& d \big((x ^2 - y ^2) \, F\big) = - \big((x ^2 - y ^2) \, \Delta x\big) \, dx, \\ &  d \big((x ^2 - y ^2) \, G \big) = - \big((x ^2 - y ^2) \, \Delta y\big) \, dy.  \end{split} \end{equation}
It follows that $(x ^2 - y ^2) \, F = A (x)$, for some (smooth) function $A$ of one variable and that  $A' (x) = - (x ^2 - y ^2) \, \Delta x$; likewise, $(x ^2 - y ^2) \, G = B (y)$ and $B ' (y) = - (x ^2 - y ^2) \, \Delta y$.
\end{proof}
A simple computation using (\ref{Deltaxy}) shows that in terms of $A, B$, the functions $h _+, h _-$ appearing in 
{\rm \eqref{nablaalphasharp}} and their derivatives $d h_+$, $d h _-$ have the following expressions:
\begin{equation} \label{h+-} \begin{split} 
& h _+ = - \frac{A' (x) + B' (y)}{2 (x ^2 - y ^2)} + \frac{(x - y) (A (x) + B (y))}{(x ^2 - y ^2) ^2}, \\
& h _- = - \frac{A' (x) -  B' (y)}{2 (x ^2 - y ^2)} + \frac{(x + y) (A (x) + B (y))}{(x ^2 - y ^2) ^2}, \end{split} \end{equation}
\begin{equation} \label{dh+} \begin{split} 
 d h _+ & = - \frac{A'' (x) dx + B'' (y) dy}{2 (x ^2 - y ^2)} \\ &  + \frac{A' (x) \big((2 x - y) \, dx - y \, dy\big) + B ' (y) \big(x \, dx + 
(x - 2 y) \, dy\big)}{(x ^2 - y ^2) ^2} \\ &  
-  \frac{\big(A (x) + B (y)\big) (x - y) \big((3 x - y) \, dx + (x - 3 y) \, dy \big)}{(x ^2 - y ^2) ^3}, \end{split} \end{equation} 
and
\begin{equation} \label{dh-} \begin{split} 
 d h _- & = - \frac{A'' (x) dx - B'' (y) dy}{2 (x ^2 - y ^2)} \\ &  
+ \frac{A' (x) \big((2 x + y) \, dx - y \, dy\big) + B ' (y) \big(- x \, dx + 
(x + 2 y) \, dy\big)}{(x ^2 - y ^2) ^2} \\ &  
-  \frac{\big(A (x) + B (y)\big) (x + y) \big((3 x + y) \, dx -  (x + 3 y) \, dy \big)}{(x ^2 - y ^2) ^3}.  \end{split} \end{equation} 
In particular:
\begin{equation} \label{Jdh} J _+  d h_+ - J _- d h _- = \left(\frac{h_+}{f_+} - \frac{h_-}{f_-}\right). \end{equation}

\begin{prop} \label{prop-killingvf} The vector fields
\begin{equation} \label{K1} \begin{split} K _1 & := J _+ {\rm grad} _g (x + y) = J _- {\rm grad} _g (x - y)  \\ & = J _+ {\rm grad} _{g _+} \left(\frac{- 1}{x + y}\right) = J _- {\rm grad} _{g_-} \left(\frac{-1}{x - y}\right) 
\end{split} \end{equation}
(which is equal to the vector field $K_1 = - \frac{1}{2} \alpha ^{\sharp}$ appearing in 
{\rm Proposition \ref{prop-killing}}), 
and
\begin{equation} \label{K2} \begin{split} K _2 & 
:= y ^2 \, J _+ {\rm grad} _g x + x ^2 \, J _+ {\rm grad} _g  y = 
y ^2 \, J _- {\rm grad} _g x -  x ^2 \, J _- {\rm grad} _g  y\\ & =
J _+ {\rm grad} _{g_+} \left(\frac{x y}{x + y}\right) = J _- {\rm grad} _{g_-} \left(\frac{- x y}{x - y}\right) \end{split} \end{equation}
are  Killing with respect to $g, g_+, g_-$ and Hamiltonian with respect to $\omega _+$ and $\omega_-$. The momenta, $\mu _1 ^+$, $\mu _2 ^+$ of $K _1, K _2$ with respect to $\omega _+$, and the momenta, $\mu _1 ^-$, $\mu _2 ^-$, of $K _1, K _2$  with respect to $\omega _-$, are  given by
\begin{equation} \label{mu} \begin{split} & \mu _1 ^+ = \frac{- 1}{x + y}, \qquad \mu _2 ^+ = \frac{x y}{x + y}, \\ & \mu _1 ^- = \frac{- 1}{x - y}, \qquad \mu _2 ^- = \frac{- x y}{x - y},  \end{split} \end{equation}
and Poisson commute with respect to $\omega _+$ and $\omega _-$,  meaning that 
$\omega _{\pm}  (K _1, K _2) = 0$, so that  $[K _1, K _2 ] = 0$ as well. 
In particular, on the open set $M _1$, the ambik\"ahler structure 
$(g_+, J_+, \omega_+)$, $(g _-, J_-, \omega _-)$ is {\it ambitoric} in the sense of {\rm \cite[Definition 3]{ACG2}}. 
\end{prop}
\begin{proof}
In terms of $A, B$, (\ref{nabladxy}) can be rewritten as
\begin{equation} \begin{split} 
& \nabla d x = \frac{1}{4 (x ^2 - y ^2) ^2} \, \Big(2 x \, \big(A (x) + B (y)\big) + (x ^2 - y ^2) \, A' (x)\Big) \, g \\ & \qquad  - \frac{1}{4 (x ^2 - y ^2) ^2} \, \Big(2 x \, \big(A (x) + B (y)\big) - (x ^2 - y ^2) \, A' (x)\Big) \, g _{\tau} \\ & \qquad 
- \frac{x}{(x ^2 - y ^2)} \, (d x \otimes d x + d y \otimes d y) + 
\frac{y}{(x ^2 - y ^2)} \, (d x \otimes d y + d y \otimes d x) \\ &  \qquad
- \frac{x}{(x ^2 - y ^2)} \, J _+ (d x + d y) \otimes J _+ (d x + d y), \\ &
 \nabla d y = \frac{1}{4 (x ^2 - y ^2) ^2} \, \Big(- 2 y \, \big(A (x) + B (y)\big) + (x ^2 - y ^2) \, B' (y)\Big) \, g\\ & \qquad  - \frac{1}{4 (x ^2 - y ^2) ^2} \, \Big(2 y \, \big(A (x) + B (y)\big) +  (x ^2 - y ^2) \, B ' (y)\Big) \, g _{\tau} \\ & \qquad 
+ \frac{y}{(x ^2 - y ^2)} \, (d x \otimes d x + d y \otimes d y) - 
\frac{x}{(x ^2 - y ^2)} \, (d x \otimes d y + d y \otimes d x) \\ & \qquad
+  \frac{y}{(x ^2 - y ^2)} \, J _+ (d x + d y) \otimes J _+ (d x + d y).
\end{split} \end{equation}
By taking (\ref{nablaJ+})--(\ref{nablaJ-}) into account, we infer 
\begin{equation} \begin{split} \label{nablaJdx} \nabla (J _+ dx) &  = 
\frac{1}{2 (x ^2 - y ^2)} \, 
\left(\frac{(2 y - x) \, A (x) + x \, B (y)}{(x ^2 - y ^2)} + \frac{A' (x)}{2}\right) \, g (J _+ \cdot, \cdot) \\ & - \frac{1}{2 (x ^2 - y ^2)} \, \left(\frac{x A (x) + x B (y)}{(x ^2 - y ^2)} - \frac{A' (x)}{2}\right) \, g (J _-  \cdot, \cdot) \\ & 
- \frac{y \, d x \wedge J _+ dx + x \,  dy \wedge J _+ d y}{(x ^2 - y ^2)} \\ & 
+ \frac{x \, (d x \otimes J _+ d y + J _+ d y \otimes dx) + y \, (d y \otimes J _+ dx + J _+ dx \otimes dy)}{(x ^2 - y ^2)}
\end{split} \end{equation}
and
\begin{equation} \begin{split} \label{nablaJdy} \nabla (J _+ dy) &  = 
\frac{1}{2 (x ^2 - y ^2)} \, 
\left(\frac{(- y \, A (x) + (y - 2 x) \, B (y)}{(x ^2 - y ^2)} + \frac{B' (y)}{2}\right) \, g (J _+ \cdot, \cdot) \\ & - \frac{1}{2 (x ^2 - y ^2)} \, \left(\frac{y \,  A (x) + y \,  B (y)}{(x ^2 - y ^2)} +  \frac{B' (y)}{2}\right) \, g (J _-  \cdot, \cdot) \\ & 
+ \frac{y \, d x \wedge J _+ dx +   x \,  dy \wedge J _+ d y}{(x ^2 - y ^2)} 
\\ & 
-  \frac{x \, (d x \otimes J _+ d y + J _+ d y \otimes dx) + y \, (d y \otimes J _+ dx + J _+ dx \otimes dy)}{(x ^2 - y ^2)}.
\end{split} \end{equation}
In particular, the symmetric parts of $\nabla (J _+ dx)$ and   
$\nabla (J _+ dy)$ are  opposite and given by 
\begin{equation} \label{nablaJdxyS} \begin{split} 
 \big(\nabla (J _+ dx)\big) ^s =  - \big(\nabla (J _+ dy)\big) ^s  = &
\frac{x \, (d x \otimes J _+ d y + J _+ d y \otimes dx)}{(x ^2 - y ^2)} \\ &  + \frac{y \, (d y \otimes J _+ dx + J _+ dx \otimes dy)}{(x ^2 - y ^2)}.  \end{split} \end{equation}
The symmetric parts of $\nabla (J _+ dx + J _+ dy)$ and of 
$\nabla (y ^2 J _+ dx + x ^2 J _+ dy) = y ^2 \nabla (J _+ dx) + x ^2 \nabla (J _+ dy) + 2 d y \otimes J _+ dx + 2 x dx \otimes J _+ dy$ 
then clearly vanish, meaning that $K_1$ and $K _2$ are Killing with respect to $g$. In view of the expressions of $K _1, K_2$   as symplectic gradients in (\ref{K1})--(\ref{K2}), $K _1$ and $K _2$ are Hamiltonian with respect to $\omega _+$ and $\omega _-$, their momenta are those given by (\ref{mu}) and their Poisson bracket with respect to $\omega _{\pm}$ is equal to $\omega _{\pm} (K_1, K_2)$, which is zero, since $ dx$ lives in the dual of $T ^+$ and $d y$ in the dual of $T ^-$. This, in turn, implies that $K _1$ and $K _2$ commute. 
\end{proof}

\begin{rem} \label{rem-jelonek} {\rm 
As already observed,   the Killing vector field $K _1$ appearing in Proposition \ref{prop-killingvf} is the restriction to $M _1$ of the smooth vector field, also denoted by $K _1$, appearing in Proposition \ref{prop-killing}, which is defined on the whole manifold $M$ by
\begin{equation} \label{K1-M} K _1 = - \frac{1}{2} \alpha ^{\sharp} = - 
\frac{1}{6} \delta \Psi. \end{equation}
Similarly, it is easily checked that $K _2$ is the restriction to $M _1$ 
of the smooth vector field, still denoted by $K _2$, defined on $M$ by
\begin{equation} \label{K2-M} \begin{split} K _2 = & - \frac{1}{8} \delta \big((f _+ ^2 - f _- ^2) \, (\Psi _+ - \Psi _-)\big) \\ & = \frac{1}{8} 
\big(\Psi  _+ - \Psi _-\big) \,  \big({\rm grad} _g (f _+ ^2  - f _- ^2)\big)
\end{split} \end{equation}
(recall that the Killing $2$-form $\varphi = \psi _+ - \psi _- = * \psi$ is co-closed).  It is also easily checked that $K _2$ and  $K _1$ are related by 
\begin{equation} \label{K1K2} K _2 = \frac{1}{2} \, S (K _1), \end{equation}
where, we recall, $S$ denotes  the {\it Killing}  symmetric endomorphism defined by (\ref{S}) in Remark \ref{rem-Skilling}; this is because, on the dense open subset $M _0$, $S$ can be rewritten as 
\begin{equation} S = - (x ^2 - y ^2) \, \tau + (x ^2 + y ^2) \, {\rm I}, \end{equation}
whereas $K _1 ^{\flat} = J _+ (d x + d y)$, so that $S (K _1 ^{\flat}) = 
2 y ^2 \, J _+ dx + 2 x ^2 \, J _+ dy = 2 K _2 ^{\flat}$; we thus get (\ref{K1K2}) on $M _0$, hence on $M$. In view of (\ref{K1K2}),  the fact that $K _2$ is Killing can then be alternatively deduced from \cite[Lemma B]{jelonek}, cf. also the proof of \cite[Proposition 11 (iii)]{ACG2}.
} \end{rem} 
In view of the above, we eventually get the following rough classification:
\begin{prop} \label{prop-f} For any connected, oriented, $4$-dimensional Riemannian manifold $(M, g)$ admitting a non-parallel $*$-Killing $2$-form $\psi$,  the open subset $M _S$ defined by 
{\rm \eqref{MS}} is either empty or dense and we have one of the following three exclusive possible cases:
\begin{enumerate}

\item $M _S$ is dense; the vector fields $K_1, K_2$ are Killing  and independent on a dense open set of $M$, or
\smallskip

\item $M _S$ is dense;   the vector fields $K_1, K_2$ are Killing and $K _2 = c \, K _1$, for some non-zero  real number $c$, or

\smallskip

\item $M _S$ is empty, i.e. $\psi$ is decomposable everywhere; then, $K _2$ is identically zero, whereas $K _1$ is non-identically zero and is not a Killing vector field in general.

\end{enumerate}

\end{prop}
\begin{proof} Being Killing on $M _0 \cap M _S$ and zero on any open set where $f _+ = f _-$, $K _2$ is Killing everywhere on $M$. We next observe that, for any ${\sf x}$ in  $M _S$, $K _2 ({\sf x}) = 0$ if and only if $K _1 ({\sf x})  = 0$, as 
readily follows from (\ref{K1K2}) and from the fact that $S$ is invertible if and only if  ${\sf x}$ belongs to $M _S$, as the eigenvalues of $S$ are equal to 
$\frac{(f _+ + f _-) ^2}{2}$ and $\frac{(f_+ - f _-) ^2}{2}$. 

Suppose now that $M _S$ is not dense in $M$, i.e. that $M \setminus M _S$ contains some  non-empty open subset 
$V$; then, $K _2$ vanishes on $V$, hence vanishes identically on $M$, 
as $K _2$ is Killing; from (\ref{K2-M}), we then infer $0 = \Psi (K _2) = \frac{1}{8} (f _+ ^2 - f _- ^2) {\rm grad} _g (f _+ ^2 - f _- ^2)$, which implies that the (smooth) function $(f _+ ^2 - f _- ^2)^2$ is {\it constant} on $M$, hence identically zero, meaning that $M _S$ is empty.  If $M _S$ is empty, then  $f _+ = f _-$ everywhere (equivalently, $\psi \wedge \psi$ is identically zero); it 
follows that $K _2$ is identically zero,  whereas $K _1$, which is not identically zero since $\psi$ is not parallel, is not Killing in general, cf. Section \ref{sdecomposable}.

If $M_S$ is dense, then $K _1$ and $K _2$ are both Killing vector fields  on $M$, hence either independent on some dense open subset of $M$ or dependent 
everywhere and, by the above discussion,  $K _2$ is then a constant, non-zero multiple of $K _1$.
\end{proof}

In the next sections we successively consider the three cases listed in Proposition \ref{prop-f}.

\section{The ambitoric Ansatz} \label{sambitoric} 

In this section, we assume that $M _S$ is dense and that $K _1, K _2$ are independent on some dense open set $\mathcal{U}$. In the remainder of this section, we focus our attention on $\mathcal{U}$, i.e. we assume that $d x$ and $d y$ are independent everywhere --- equivalently, {\rm $\tau(df)  \neq \pm df$} 
everywhere --- so that $\{ dx, J _+ d x = J _- dx, dy, J _+ dy = - J _- dy \}$ form a direct  orthogonal coframe. By Proposition \ref{prop-separate}, the metric $g$ and the K\"ahler forms $\omega _+$, $\omega _-$ can then be written as
\begin{equation} \begin{split} \label{g-prim}g & = (x ^2 - y ^2) \, \left(\frac{d x \otimes d x}{A (x)} + \frac{d y \otimes dy}{B (y)}\right) \\ & 
+ (x ^2 - y ^2) \, \left(\frac{J _+ d x \otimes J _+ d x}{A (x)} + \frac{J_+ d y \otimes J _+ dy}{B (y)}\right), \end{split} \end{equation}
\begin{equation} \label{omega-prim}
\begin{split} &\omega _+ = \frac{(x - y)}{(x + y)}  \, \left(\frac{d x \wedge J _+ dx}{A (x)} + \frac{d y \wedge J _+ dy}{B (y)}\right), \\ &
\omega  _- = \frac{(x + y)}{(x - y)}  \, \left(\frac{d x \wedge J _+ dx}{A (x)} -  \frac{d y \wedge J _+ dy}{B (y)}\right), \end{split} \end{equation}
and we also have:
\begin{prop} The functions ${\rm Scal} = 4 a$ and $b$ appearing in the expression {\rm (\ref{Ric})}  of the Ricci tensor of $g$ are given by:
\begin{equation} \label{scal-AB} {\rm Scal} = - \frac{A'' (x) + B '' (y)}{(x ^2 - y ^2)}, \end{equation}
and
\begin{equation} \label{b-AB} b = - \frac{A'' (x) -  B''(y)}{4 (x ^2 - y ^2)} + \frac{x A' (x) + y B ' (y)}{(x ^2 - y ^2)^2} - \frac{A (x) + B (y)}{(x^2 - y^2)^2}. \end{equation}
\end{prop}
\begin{proof} Since $\alpha ^{\sharp}$ is Killing, the Bochner formula reads:
\begin{equation} \label{bochner} {\rm Ric} (\alpha ^{\sharp}) = \delta \nabla \alpha ^{\sharp} \end{equation}
whereas, by (\ref{Ric}), 
\begin{equation} {\rm Ric} (\alpha ^{\sharp}) = a \, \alpha ^{\sharp} + b \, \tau (\alpha ^{\sharp}). \end{equation}
By using
\begin{equation} \label{alphadeltaJ} \alpha = f _+ \, \delta J _+ = f _- \, \delta J _-, \end{equation} 
which easily follows from (\ref{nablaJ+})--(\ref{nablaJ-}), we infer from (\ref{nablaalphasharp}) that
\begin{equation} \label{deltanablaalpha} \delta \nabla \alpha ^{\sharp} = \frac{h _+}{f_+} \, \alpha + \frac{h_-}{f_-} \, \alpha - J _+ dh _+ - J _- d h _-. \end{equation}
By putting together (\ref{bochner}), (\ref{deltanablaalpha})
and (\ref{Jdh}), we get
\begin{equation} a \, \alpha + b \, \tau (\alpha) = 2 \left(\frac{h_+}{f_+} \, \alpha - J _+ d h _+ \right) = 2 \left(\frac{h_-}{f_-} \, \alpha - J _- d h _-\right),  \end{equation}
hence
\begin{equation} \label{dh} \begin{split} & 
d h _+ = \left(a - 2 \frac{h_+}{f_+} + b\right) \, d x 
+ \left(a - 2 \frac{h_+}{f_+} - b\right) \, d y, 
\\ &
dh _- = \left(a - 2 \frac{h_-}{f_-} + b\right) d x 
+ \left(- a + 2 \frac{h_-}{f_-} + b\right) d y. \end{split} \end{equation}
We thus get
\begin{equation} \begin{split} 
& a = \frac{1}{2} \left(\frac{\partial h _+}{\partial x} + \frac{\partial h_+}{\partial y}\right) + \frac{2 h _+}{x + y} = 
\frac{1}{2} \left(\frac{\partial h _-}{\partial x} - \frac{\partial h_-}{\partial y}\right) + \frac{2 h _-}{x + y}
\\ &
b = \frac{1}{2} \left(\frac{\partial h _+}{\partial x} -  \frac{\partial h_+}{\partial y}\right) = \frac{1}{2} \left(\frac{\partial h _-}{\partial x} +  \frac{\partial h_-}{\partial y}\right).
 \end{split} \end{equation}
By using (\ref{h+-}), we obtain (\ref{scal-AB}) and (\ref{b-AB}). 
\end{proof}

Recall that a  function $\varphi$ is called $J _+$-{\it pluriharmonic} if $d (J _+ d \varphi) = 0$ and $J _-$-{\it pluriharmonic} if $d (J _- d \varphi) = 0$.
\begin{prop} \label{prop-t} {\rm (i)} Up to a multiplicative and an additive constant, the function 
\begin{equation} \label{varphi+} \varphi _+ = \int ^x \frac{dt}{A (t)} - \int ^y \frac{dt}{B (t)} \end{equation}
is the only $J _+$-pluriharmonic function of the form $\varphi = \varphi (x, y)$. 
\smallskip

{\rm (ii)} Up to a multiplicative and an additive constant, the function 
\begin{equation} \label{varphi-} \varphi _- = \int ^x \frac{dt}{A (t)} +  \int ^y \frac{dt}{B (t)} \end{equation}
is the only $J _-$-pluriharmonic function of 
the form $\varphi = \varphi (x, y)$. 
\end{prop}
\begin{proof} 
From (\ref{nablaJdx})--(\ref{nablaJdy}), we readily infer the following expression of $d (J_{\pm} d x)$ and $d (J _{\pm} dy)$:
\begin{equation} \label{ddcx} \begin{split} 
d (J _+ d x)  = d (J _- dx) & =   \left(\frac{A' (x)}{A (x)} - \frac{2 x}{x ^2 - y ^2}\right) \, d x \wedge J _+ dx \\ &    + \frac{2 y \, A (x)}{(x ^2 - y ^2) \, B (y)} \, d y \wedge J _+ dy,  \end{split} \end{equation}
and
\begin{equation} \label{ddcy} \begin{split}  
d (J _+ d y) = - d (J _- dy) = &   - \frac{2 x \, B (y)}{(x ^2 - y ^2) \, A (x)} \, d x \wedge J_+ dx \\ &   + \left(\frac{B ' (y)}{B (y)} + \frac{2 y}{x ^2 - x ^2}\right) \, dy \wedge J _+ d y. 
\end{split} \end{equation}
Let $\varphi = \varphi (x, y)$ be any function of $x, y$ and denote by $\varphi _x, \varphi _y, \varphi _{x x}, etc...$ its derivative with respect to $x, y$. Then
\begin{equation} \begin{split}
d (J _+ d \varphi) & = \varphi _x \, d (J _+ d x)  + \varphi _y \, d (J _+ dy) \\ & + \varphi _{x x} \, d x \wedge J_+ d x + \varphi _{y y } \, d y \wedge J _+ d y \\ & + \varphi _{x y} \, (d x \wedge J _+ dy + d y \wedge J _+ d x). \end{split} \end{equation}
By (\ref{ddcx})--(\ref{ddcy}), $\varphi$ is $J _+$-pluriharmonic if and only if $\varphi _{x y} = 0$ --- meaning that $\varphi$ is of the form
$\varphi (x, y) = C (x) + D (y)$ --- 
and $C, D$ satisfy 
\begin{equation} \begin{split} 
& C'' (x) + \left(\frac{A' (x)}{A (x)} - \frac{2 x}{x ^2 - y ^2}\right)  C' (x) - \frac{2 x \, B (y) \, D ' (y)}{(x ^2 - y ^2) \, A (x)} = 0, \\ & 
D'' (y) + \left(\frac{B' (y)}{B (y)} + \frac{2 y}{x ^2 - y ^2}\right)  
D' (y) +  \frac{2 y \, A (x) \, C ' (x)}{(x ^2 - y ^2) \, B (y)} = 0.
\end{split} \end{equation}
It is easily checked that the pair $C' (x) = \frac{k}{A (x)}, D' (y) = - \frac{k}{B (y)}$, for some constant $k$,  is the unique solution to this system. We thus get (\ref{varphi+}). We check (\ref{varphi-}) similarly.
\end{proof}
In view of Proposition  \ref{prop-t}, we (locally) define $t$, up to an additive constant, by
\begin{equation} J _+ d \varphi _+ = J _- d \varphi _- = -  d t,  
\end{equation} 
and we denote by $\eta$ the $1$-form defined by 
$\eta = - \frac{\tau (d t)}{2}$. 
We then have
\begin{equation} d t = 
- \frac{J _+ d x}{A (x)} +  \frac{J _+ d y}{B (y)}, \quad \eta = \frac{1}{2} \left(\frac{J _+ d x}{A (x)} + \frac{J _+ d y}{B (y)}\right),   \end{equation}
hence
\begin{equation} \label{defJ+} \begin{split} & J _+ d x = J _- d x = A (x) \, 
(\eta -  \frac{d t}{2}), \\ & 
J _+ d y = - J _- d y = B (y)\, (\eta + \frac{d t}{2}).
\end{split} \end{equation}
Notice that
\begin{equation} \label{dxdydteta} dx \wedge dy \wedge \eta \wedge dt
= \frac{v _g}{(x ^2 - y ^2) ^2}, \end{equation}
where $v _g$ denotes the volume form of $g$ with respect to the orientation induced by $J _+$. 
By using (\ref{ddcx})--(\ref{ddcy}), then  (\ref{defJ+}), we get
\begin{equation} \label{deta} \begin{split} d \eta & =  
\frac{1}{(x ^2 - y ^2)} \, \left(\frac{- 2 x \, dx \wedge J _+ dx}{A (x)}  +  \frac{2 y \, d y \wedge J _+ dy}{B (y)}\right) \\ & 
=  \frac{1}{(x ^2 - y ^2)} \, \big(- 2 x d x \wedge (\eta -  \frac{d t}{2}) + 2 y d y \wedge (\eta + \frac{dt}{2})\big) \\ & = 
\frac{1}{(x ^2 - y ^2)} \, \big(- (2 x dx -  2 y dy) \wedge \eta + (x dx +  y dy) \wedge dt\big). \end{split} \end{equation}
It follows that $(x ^2 - y ^2) \, \eta - (x ^2 + y ^2) \, \frac{dt}{2}$ is closed, hence locally defines a function $s$ by
\begin{equation} \label{ds} (x ^2 - y ^2) \, \eta -  (x ^2 + y ^2) \, \frac{dt}{2} = d s, \end{equation} 
hence
\begin{equation} \label{ds-bis} d s = \frac{x ^2 \, J _+ dx}{A (x)} - \frac{y ^2 \, J _+ d y}{B (y)}. \end{equation}
We thus have:
\begin{equation}  \eta -  \frac{d t}{2} = 
\frac{d s + y ^2 \, dt}{(x ^2 - y ^2)}, \quad  
\eta +  \frac{d t}{2} = 
\frac{d s + x ^2 \, dt}{(x ^2 - y ^2)}, \end{equation}
whereas (\ref{dxdydteta}) can be rewritten as
\begin{equation} dx \wedge dy \wedge d s  \wedge d t = \frac{v _g}{(x ^2 - y ^2) }, \end{equation}
showing  that $d x, d y, d s, d t$ form a (direct) coframe.
In view of (\ref{g-prim}), (\ref{omega-prim}), (\ref{defJ+}), on the open set where $x, y, s, t$ form a coordinate system, the metrics $g, g _+, g _-$, the complex structures $J_+, J_-$, the involution $\tau$  and the K\"ahler forms $\omega _+, \omega _-$ have the following expressions:
\begin{equation} \label{A-g} \begin{split} g & = (x ^2 - y ^2) \, \left(\frac{d x \otimes d x}{A (x)} + \frac{d y \otimes d y}{B (y)}\right) \\ & \quad 
+ \frac{A (x)}{(x ^2 - y ^2)} \, (d s + y ^2 \, dt) \otimes 
(d s +  \, y ^2 \, dt) \\ & \quad 
+ \frac{B (y)}{(x ^2 - y ^2)} \, (d s + x ^2 \, dt) \otimes 
(d s + x ^2 \, dt)\\ & = (x + y) ^2 \, g _+ = (x - y) ^2 \, g _- \end{split} \end{equation}
\smallskip

\begin{equation} \label{A-J} \begin{split} & J _+ dx = J _- dx = \frac{A (x)}{(x ^2 - y ^2)} \, (d s + y ^2 \, dt) \\ & 
J _+ dy = - J _- dy = \frac{B (y)}{(x ^2 - y ^2)} \, (d s + x ^2 \, dt) \\ &
J _+ dt = \frac{d x}{A (x)} - \frac{d y}{B (y)}, \quad 
J _- dt = \frac{d x}{A (x)} + \frac{d y}{B (y)} \\ &
J _+ d s = - \frac{x ^2 \, dx}{A (x)} + \frac{y ^2 \, dy}{B (y)}, \quad J _- d s = - \frac{x ^2 \, dx}{A (x)} - \frac{y ^2 \, dy}{B (y)}
\end{split} \end{equation}

\smallskip

\begin{equation} \label{A-tau} \begin{split} & \tau (d x) = dx, \qquad \tau (d y) = - d y \\ &
\tau (d s) = \frac{(x ^2 + y ^2)}{(x ^2 - y ^2)} \, d s + \frac{2 x ^2 y ^2}{(x ^2 - y ^2)} \, dt \\ & \tau (d t) = \frac{- 2}{(x ^2 - y ^2)} \, ds  - \frac{(x ^2 + y ^2)}{(x ^2 - y ^2)} \, dt, \end{split} \end{equation}
\smallskip

\begin{equation} \label{A-omega} \begin{split} & \omega _+ =  
\frac{d x \wedge (d s + y ^2 \, dt) +
 d y \wedge (d s + x ^2 \, dt)}{(x + y) ^2} \\ 
& \omega _- = \frac{d x \wedge (d s + y ^2 \, dt) - dy \wedge (d s + x ^2 \, dt)}{(x - y) ^2} \end{split} \end{equation}
\smallskip

whereas, it follows from (\ref{psi}) that  the $*$-Killing $2$-form $\psi$ is given by

\begin{equation} \label{A-psi}   \psi  = 
2 x \, d x \wedge (d s + y ^2 \, dt) +
2 y   \, dy \wedge (d s + x ^2 \, dt).  \end{equation}
Notice that, in view of (\ref{A-g}), the (local) vector fields 
$\frac{\partial}{\partial s}$ and $\frac{\partial}{\partial t}$ are Killing with respect to $g$ and respectively coincide with the Killing vector fields $K _1$ and $K _2$ appearing in Proposition 
\ref{prop-killingvf} on their domain of definition.  

\smallskip

It turns out that the expressions of $(g _+ = (x + y) ^{-2} \, g, J_+, \omega _+)$ and $(g _- = (x - y) ^{-2} \, g, J_-, \omega _-)$ just obtained coincide with the {\it ambitoric Ansatz} described in \cite[Theorem 3]{ACG2}, in the case  
 where the quadratic polynomial is $q (z) = 2 z$, 
which is the {\it normal form}  of the ambitoric Ansatz in the {\it hyperbolic} case considered in \cite[Paragraph 5.4]{ACG2}.

 \smallskip

The discussion in this section  can then be summarized as follows: 
\begin{thm} \label{thm-ambitoric} Let $(M, g)$ be a connected, oriented, $4$-dimensional manifold admitting a non-parallel,  $*$-Killing $2$-form $\psi = \psi _+ + \psi _-$ and assume that the open set, $M _S$,  where $|\psi _+| \neq |\psi _-|$ is dense, cf. {\rm Proposition \ref{prop-f}}.  On the open subset, $\mathcal{U}$, of $M _S$ where  $\psi _+$ and $\psi _-$ have no zero and $d |\psi _+| \wedge d |\psi _-| \neq 0$, the pair $(g, \psi)$ gives rise to an ambitoric structure of hyperbolic type, in the sense of {\rm \cite{ACG2}}, relative to the conformal class of $g$, which, on any simply-connected open subset of $\mathcal{U}$,    is  described by 
{\rm (\ref{A-g})--(\ref{A-J})--(\ref{A-omega})}, where the Hermitian structures 
$(g _+ = (x + y) ^{-2} \, g, J _+, \omega _+)$ and $(g _- = (x - y) ^{-2} g, 
J _-, \omega _-)$ are K\"ahler,  whereas $\psi$ is described by {\rm (\ref{A-psi})}. 

Conversely, on the open set,  $\mathcal{U}$,  of $\mathbb{R} ^4$, of coordinates $x, y, s, t$, with $x > |y| > 0$, the  two almost Hermitian structures 
$(g _+ = (x + y) ^{-2} \, g, J _+, \omega _+)$, 
$(g _- = (x - y) ^{-2} g, J _-, \omega _-)$ defined by {\rm (\ref{A-g})--(\ref{A-J})--(\ref{A-omega})}, with $A (x) > 0$ and $B (y) > 0$,  are    
K\"ahler and, together with the Killing vector fields $K _1 = \frac{\partial}{\partial s}$ and $K _2 = 
\frac{\partial}{\partial t}$, constitute an ambitoric structure of hyperbolic type, whereas the $2$-form $\psi$ defined by {\rm (\ref{A-psi})} is $*$-Killing with respect to $g$. 
\end{thm}
\begin{proof} The first part results of the preceding discussion. 
For the converse, we first observe that 
the $2$-forms $\omega _+$ and $\omega _-$ defined by (\ref{A-omega}) are clearly closed and not degenerate.  To test the integrability of  the almost complex structures  $J _+$ and $J _-$ defined by (\ref{A-J}), we consider 
the complex $1$-forms: 
\begin{equation} \begin{split} & \beta _+  = d x + i \, J _+ dx = d x + i \frac{A (x)}{(x ^2 - y ^2)} \, (d s +  y ^2 \, dt), \\ & 
\gamma _+  = d y + i \, J _+ d y = d y + i \frac{B (y)}{(x ^2 - y ^2)} \, 
(d s + x ^2 \, dt),  \end{split} \end{equation}
which generate the space of $(1, 0)$-forms with respect to $J _+$. 
We then have:  
\begin{equation} \begin{split} d \beta _+&  = i \, \frac{\big((x ^2 - y ^2) \, A ' (x) + x \, A (x)\big)}{(x ^2 - y ^2)} \, d x \wedge (d s + y ^2 \, dt) \\ & 
\quad  + i \, \frac{2 y \, A (x)}{(x ^2 - y ^2)} \, d y \wedge (d s + x ^2 \, dt)
 \\ & \quad = \frac{\big(A' (x) - 2 x \, A (x)\big)}{A (x)} \, d x \wedge \beta _+ + \frac{2 y \, A (x)}{B (y)} \, dy \wedge \gamma _+, \\ 
d \gamma _+  & = i \, \frac{\big((x ^2 - y ^2) \, B ' (y) + 2 y  \, B (y)\big)}{(x ^2 - y ^2)}  \, d y \wedge (d s + x ^2 \, dt) \\ & \quad  - i \, \frac{2 x \, B (y)}{(x ^2 - y ^2)} \, d x \wedge (d s + y ^2 \, dt) \\ & \quad = 
\frac{\big(B' (y) + 2 y \, B (y)\big)}{B (y)} \, dy \wedge \gamma _+ - \frac{2 x \, B (y)}{A (x)} \, d x \wedge \beta _+,
\end{split} \end{equation}
which shows that $J _+$ is integrable. For $J _-$, we likewise consider 
the complex $1$-forms: 
\begin{equation} \begin{split} & \beta _- = d x + i \, J _- dx = \beta _+ = d x + i \frac{A (x)}{(x ^2 - y ^2)} \, (d s + y ^2 \, dt), \\ & 
\gamma _-  = d y + i \, J _- d y = d y -  i \frac{B (y)}{(x ^2 - y ^2)} \, 
(d s + x ^2 \, dt),  \end{split} \end{equation}
which generate the space of $(1, 0)$-forms with respect to $J _+$. We then get
\begin{equation} \begin{split} d \beta _- &  = d \beta _+ 
 \\ & \quad = \frac{\big(A' (x) - 2 x \, A (x)\big)}{A (x)} \, d x \wedge \beta _- -   \frac{2 y \, A (x)}{B (y)} \, dy \wedge \gamma _-, \\ 
d \gamma _-  & = - d \gamma _+  \\ & \quad = 
\frac{\big(B' (y) + 2 y \, B (y)\big)}{B (y)} \, dy \wedge \gamma _- 
+ \frac{2 x \, B (y)}{A (x)} \, d x \wedge \beta _-,
\end{split} \end{equation}
which, again,  shows that $J _-$ is integrable. It follows that the almost Hermitian structures $(g _+ = (x + y) ^{-2} \, g, J_+, \omega _+)$ and 
$(g _- = (x - y) ^{-2} \, g, J_-, \omega _-)$ 
are  both {\it K\"ahler} and thus determine an {\it ambik\"ahler structure} on $\mathcal{U}$. Moreover, the vector fields $\frac{\partial}{\partial s}$ and 
$\frac{\partial}{\partial t}$ are clearly Killing with respect to $g, g_+, g _-$, and satisfy:
\begin{equation} \begin{split} 
& \frac{\partial}{\partial s} \lrcorner \omega _+ = - \frac{d x + d y}{(x + y) ^2} = d \left(\frac{1}{x + y}\right), \quad 
\frac{\partial}{\partial s} \lrcorner \omega _- = 
\frac{- d x + d y}{(x - y) ^2} = d \left(\frac{1}{x - y}\right), \\ &
\frac{\partial}{\partial t} \lrcorner \omega _+ = 
- \frac{y ^2 \, d x + x ^2 \, d y}{(x + y) ^2} = - 
d \left(\frac{x y}{x + y}\right), \quad
\frac{\partial}{\partial t} \lrcorner \omega _- = 
- \frac{y ^2 \, d x - x ^2 \, d y}{(x - y) ^2} =  
d \left(\frac{x y}{x - y}\right), \end{split} \end{equation}
meaning that they are both Hamiltonian with respect to $\omega _+$ and $\omega _-$, with momenta given by (\ref{mu}) in Proposition \ref{prop-killingvf}.
This implies that $\frac{\partial}{\partial s}$ and $\frac{\partial}{\partial t}$ preserve the two K\"ahler structures
$(g _+, J_+, \omega _+)$ and $(g_-, J_-, \omega _-)$ and actually coincide with the vector field $K _1$ and $K _2$ respectively defined  in a more general context in Proposition \ref{prop-killingvf}.  We thus end up with an {\it ambitoric structure}, as defined in \cite{ACG2}. 
According to Theorem 3 in \cite{ACG2}, it is an 
{\it ambitoric structure  of hyperbolic type}, with ``quadratic polynomial'' $q (z) = 2 z$. To check that the $2$-form $\psi$ defined by (\ref{A-psi}) --- which is evidently closed --- is $*$-Killing with respect to $g$, denote by $f _+, f _-$ the positive functions on $\mathcal{U}$ defined by $f _+ = x + y$, $f _- = x - y$, so that $g _+ = f _+ ^{-2} \, g$, $g _- = f _- ^{-2} \, g$ and $\psi = f _+ ^3 \, \omega _+ + f ^3 _- \, \omega _-$; it then follows from (\ref{A-tau}) that $\tau (d f _+) = d f _-$, hence that $\psi$ is $*$-Killing by Proposition \ref{prop-ambikaehler}.
\end{proof} 

\section{Ambik\"ahler  structures of Calabi type} \label{scalabi}

The second case listed in Proposition \ref{prop-f}, which is considered in this section, can be made more explicit via the following proposition: 
\begin{prop} \label{prop-calabi} Let $(M, g)$ be a connected, oriented, Riemannian $4$-manifold admitting a non-parallel $*$-Killing $2$-form $\psi = \psi _+ + \psi _-$. In view of {\rm Proposition \ref{prop-f}}, assume that the open set $M _S$ --- where $\psi$ is non-degenerate --- is dense in $M$ and that the Killing vector fields $K _1, K _2$ defined by {\rm (\ref{K1-M})--(\ref{K2-M})} are related by $K _2 = c \, K _1$, for some non-zero real number $c$. Then, $c$ is positive and one of the following three cases occurs:
\begin{enumerate} 

\item $f _+ ({\sf x})  + f _- ({\sf x}) = 2 \sqrt{c}$, for any ${\sf x}$ in $M$, or

\item $f _+ ({\sf x}) - f _- ({\sf x}) = 2 \sqrt{c}$, for any ${\sf x}$ in $M$, or

\item $f _- ({\sf x})  - f _+ ({\sf x}) = 2 \sqrt{c}$, for any ${\sf x}$ in $M$,

\end{enumerate}
with the usual notation:  $f _+ = |\psi _+|/
\sqrt{2}$ and $f _- = |\psi _-|/\sqrt{2}$. 

\end{prop} 
\begin{proof} First recall that $(\Psi _+ + \Psi _-) \circ  (\Psi _+ - \Psi _-) = - (f _+ ^2 - f _- ^2) \, {\rm I}$. From (\ref{K2-M}) and $K _1 = J _+ {\rm grad} _g f _+ = J _- {\rm grad} _g f _-$, we then infer
\begin{equation} \begin{split} & \Psi (K _1) = - \frac{1}{2} {\rm grad} _g \left(f _+ ^2 + f _- ^2\right), \\ &
\Psi (K _2) = - \frac{1}{16} {\rm grad} _g 
\left(\big(f _+ ^2 - f _- ^2\big) ^2\right). \end{split} \end{equation}
On $M _S$, where $\Psi$ is invertible, the identity $K _2 = c \, K _1$ then reads:
\begin{equation} (f _+ ^2 - f _- ^2) d (f _+ ^2 - f _- ^2) = 4 c \, \big(d f _+ ^2 + d f _- ^2\big), \end{equation}
or, else:
\begin{equation} (f _+ ^2 - f _- ^2 - 4 \, c) \, d f _+ ^2 = (f _+ ^2 - f _- ^2 + 4 \, c) \, d f _- ^2. \end{equation}
Since $|d f _+| = |d f _-|$ on $M _0$, on $M _1 = M _0 \cap M _S$ we then get:
\begin{equation} f _+ ^2 (f _+ ^2 - f _- ^2 - 4 \, c) ^2 -  f _- ^2 (f _+ ^2 - f _- ^2 + 4 \, c) ^2 = 0.   \end{equation}
Since $M _1$ is dense this identity actually holds on the whole manifold $M$. It can be rewritten as
\begin{equation} (f _+ ^2 - f _- ^2) \big((f _+ + f _-) ^2 - 4 \, c \big) \big((f _+ - f _-) ^2 - 4 \, c \big) = 0; \end{equation}
this forces $c$ to be positive --- if not, $f _+ ^2 - f _- ^2$ would be identically zero --- and we eventually get the 
identity:
\begin{equation} \label{calabi-identity} \begin{split} (f _+ ^2 - f _- ^2)  (f _+ + f _- + 2 \sqrt{c}) 
(f _+ + f _- - 2 \sqrt{c}) (f _+ - f _- - 2 \sqrt{c}) (f _+ - f _- + 2 \sqrt{c}) = 0. \end{split} \end{equation}
Denote by $\tilde{M}$ the open subset of $M$ obtained by removing the zero locus $K _1 ^{-1} (0)$ of $K _1$ from $M$ (notice that $\tilde{M}$ is a connected, 
dense open subset of $M$, as $K _1 ^{-1} (0)$ is a disjoint union of totally geodesic submanifolds of codimension a least $2$).  It readily 
follows from (\ref{calabi-identity}) that $\tilde{M}$ is the union of the following four  {\it closed} subsets $\tilde{F} _0 := F _0 \cap \tilde{M}$, 
 $\tilde{F} _+ := F _+ \cap \tilde{M}$, $\tilde{F} _- := F _- \cap \tilde{M}$ and $\tilde{F} _S := F _S \cap \tilde{M}$ of $\tilde{M}$, where $F _0, F _+, F _-, F_S$ denote the four  closed subsets of $M$  defined by:
\begin{equation} \begin{split} & F _0 := \{{\sf x} \in M \, | \, f _+ ({\sf x}) + f _- ({\sf x}) = 2 \sqrt{c}\}, \\ & 
F _+ := \{{\sf x} \in M \, | \, f _+ ({\sf x}) - f _- ({\sf x}) = 2 \sqrt{c}\},  \\ &
F _- := \{{\sf x} \in M \, | \, f _- ({\sf x}) - f _+ ({\sf x}) = 2 \sqrt{c}\}, \\ & 
F _S := \{{\sf x} \in M \, | \, f _+ ({\sf x}) - f _- ({\sf x}) = 0\}. 
\end{split} \end{equation}
We now show that if the interior, $V$, of $\tilde{F} _0$ is non-empty then 
 $ \tilde F _0 =  \tilde M$ (and thus $F_0=M$ by density); this amounts to showing that the boundary $B := \bar{V} \setminus V$ of $V$ in $\tilde{M}$ is empty. If not, let ${\sf x}$ be any  element of $B$; then, ${\sf x}$ belongs to $\tilde{F} _0$, as $\tilde{F} _0$ is closed, and it also belongs to $\tilde F _+$ or $\tilde F _-$: otherwise, there would exist  an open neighbourhood of ${\sf x}$ disjoint from $\tilde F _+ \cup \tilde F _-$, hence contained in $\tilde F _0\cup \tilde F_S$; as $\tilde F _S$ has no interior, this neighbourhood would be contained in contained in $\tilde F _0$, which contradicts the fact that ${\sf x}$ sits on the boundary of $V$. Without loss, we may thus assume that ${\sf x}$ belongs to $ \tilde F _+$, so that $f _+ ({\sf x}) = 2 \sqrt{c}$ and $f _- ({\sf x}) = 0$; since $K _1 ({\sf x}) \neq 0$ --- by the very definition of $\tilde{M}$ --- $f _+$ is regular at ${\sf x}$, implying that the locus of $f _+ = 2 \sqrt{c}$ is a smooth hypersurface, $S$,  of $\tilde{M}$ near ${\sf x}$; moreover, since $\tilde{F} _+$ and $\tilde{F} _-$ are disjoint, $f _- = 0$ on $S$, meaning that $\Psi _- = 0$ on $S$; for any $X$ in $T _{\sf x} S$ we then have $\nabla _X \Psi _- = 0$. On the other hand, $\nabla _X \Psi =  (\alpha ({\sf x}) \wedge X) _-$, for any $X$ in $T _{\sf x} M$, cf. (\ref{nablaA+-}), and we can then choose $X$ in $T _{\sf x} S$ in such a way that $(\alpha ({\sf x}) \wedge X) _-$ be non-zero, 
hence $\nabla _X \psi _- \neq 0$, 
contradicting the previous assertion. We similarly show that $M = F _+$ or $M = F _-$ whenever the interior of $\tilde{F} _+$ or of $\tilde{F} _-$ is non-empty. 
\end{proof}
A direct consequence of Proposition \ref{prop-calabi} is that on the (dense) open subset $M _0$, the associated ambik\"ahler structure $(g _+ = f _+ ^{-2} \, g, J _+ = f _+ ^{-1} \, \Psi _+, \omega _+)$, 
$(g _- = f _- ^{-2} \, g = f ^2 \, g _+,  J _- = f _- ^{-1} \, \Psi _-, \omega _-)$, with $f = f _+/f _-$, satisfies 
\begin{equation} \tau (df) = - df \end{equation}
in the first case listed in Proposition \ref{prop-calabi}, and
\begin{equation} \tau (df) = df  \end{equation}
in the remaining two cases. The ambik\"ahler structure is then {\it of Calabi type}, according to the following definition, taken from \cite{ACG1}:

\begin{defi} \label{defi-calabi} {\rm An ambik\"ahler structure $(g_+, J _+, \omega_+)$, $(g_-, J_-, \omega_-)$, with $g _+ = f ^{-2} \, g _-$, 
is said to be {\it of Calabi type} if $df\ne 0 $ everywhere, and if there exists a non-vanishing  
vector field $K$, Killing  with respect to $g_+$ and $g _-$ and  Hamiltonian with respect to $\omega _+$ and $\omega _-$, which satisfies 
\begin{equation} \tau (K) = \pm \, K, \end{equation}
with $\tau = - J _+ J _- = - J _- J _+$. 
} \end{defi}
By  replacing the pair $(J _+, J _-)$ by the pair $(J_+, - J_-)$ if needed, we can assume, without loss of generality,  that $\tau (K) = - K$. In the following proposition, we recall some general facts concerning this class of ambik\"ahler structures, cf. e.g. \cite[Section 3]{ACG1}:

\begin{prop} \label{prop-K} For any ambik\"ahler structure of Calabi type, with $\tau (K) = - K$:

{\rm (i)} The Killing vector field $K$ is an eigenvector of the Ricci tensor, ${\rm Ric} ^{g_+}$, of $g _+$ and of the Ricci tensor, ${\rm Ric} ^{g_-}$, of $g_-$;
in particular, ${\rm Ric} ^{g_+}$ and ${\rm Ric} ^{g_-}$ are both $J _+$- and $J _-$-invariant;  
\smallskip

{\rm (ii)} the Killing vector field $K$ is a constant multiple of $J _- {\rm grad} _{g_-} f = J _+{\rm grad} _{g_+} \frac{1}{f}$.
\end{prop}
\begin{proof} By hypothesis, $K = J _+ {\rm grad}_{g_+}  z _+ = J _- {\rm grad} _{g_-} z _-$, for some real functions $z _+$ and $z _-$. Since $J_- K = - J_+ K$,  we infer ${\rm grad} _{g_+} z _+ = - {\rm grad} _{g_-} z _-$, hence
\begin{equation} d z _+ = - f ^{-2} \, d z _-. \end{equation}
Since $ df \neq 0$ everywhere, this, in turn,  implies that
\begin{equation} z _+ = F (f), \qquad z _- = G (f)  \end{equation}
for some real (smooth) functions $F, G$ defined on $\mathbb{R} ^{> 0}$ up to an additive constant and satisfying: 
\begin{equation} \label{Gprime} G' (x) = - x ^2 \, F' (x). \end{equation} 
Moreover, 
\begin{equation} \label{taudf-} \tau (df) = - df. \end{equation}
Since $K$ has no zero and satisfies $\tau (K) = - K$, we have
\begin{equation}  \label{J+J-} J _+ = \frac{K ^{\flat} \wedge J _+ K}{|K| ^2} + * 
 \frac{K ^{\flat} \wedge J _+ K}{|K| ^2}, \quad
J _- = - \frac{K ^{\flat} \wedge J _+ K}{|K| ^2} +  * 
 \frac{K ^{\flat} \wedge J _+ K}{|K| ^2}, \end{equation}
so that
\begin{equation} \label{J+J} J _+ -  J _- = \frac{2 \, K ^{\flat} \wedge J _+ K}{|K| ^2}, \end{equation}
In (\ref{J+J-})--(\ref{J+J}), the dual $1$-form $K ^{\flat}$ and the square norm $|K| ^2$ are relative to any metric in $[g_+] = [g_-]$. For definiteness however, we agree that they are both relative to $g_+$. Since $g ^+ = f ^{-2} \, g_-$, we have:
\begin{equation} \nabla ^{g _+} _X J _- = J _- \frac{df}{f} \wedge X + \frac{d f}{f} \wedge J_- X. \end{equation}
By using (\ref{taudf}), we then infer from  (\ref{J+J}):
\begin{equation} \label{nabla+J+J} \begin{split} \nabla ^{g_+} _X (J _+ - J_-) & = - \nabla ^{g_+} _X J _- = J _+ \frac{df}{f} \wedge X -  \frac{df}{f} \wedge J_- X \\ & = \frac{2 \, \nabla ^{g_+} _X K ^{\flat} \wedge J _+ K + 2 \, K ^{\flat} \wedge J _+ \nabla ^{g_+} _X K}{|K| ^2} \\ & - \frac{X \cdot |K| ^2}{|K| ^2} \, (J_+ + J _-).
\end{split} \end{equation} 
By contracting with $K$, and by using $K ^{\flat} = F' \, J _+ df$ and $J _+ \nabla ^{g_+} _X K = \nabla _{J_+ X} K$ (as $K$ is $J _{\pm}$-holomorphic), we obtain
\begin{equation} \label{nabla+XK} \begin{split} \nabla ^{g_+} _X K & = - \frac{|K| ^2}{2 f \, F'} J_+ X + \frac{1}{2 f \, F'} \, \left(K ^{\flat} \wedge J _+ K \right) (X) \\ & 
+ \frac{1}{2} \frac{d |K| ^2}{|K| ^2} (X) \, K  + \frac{1}{2} \frac{J _+ d |K| ^2}{|K| ^2} (X) \, J _+ K. \end{split} \end{equation}
Since $K$ is Killing with respect to $g_+$, $\nabla ^{g_+} K$ is anti-symmetric; in view of (\ref{nabla+XK}), this forces $|K| ^2$ to be of the form
\begin{equation} \label{H} |K| ^2 = H (f), \end{equation}
for some (smooth) function $H$ from $\mathbb{R} ^{> 0}$ to $\mathbb{R} ^{> 0}$, hence
\begin{equation} \label{dKnorm} \frac{d |K| ^2}{|K| ^2} = \frac{H' (f)}{H (f)} \, df = - \frac{H' (f)}{H (f) F' (f)} \, J _+ K ^{\flat}. \end{equation} By substituting 
(\ref{dKnorm}) in \eqref{nabla+XK}, we eventually get the following expression of $\nabla ^{g_+} K$:
\begin{equation} \label{nabla+K} \nabla ^{g_+} K = \Phi _+ (f) 
\, J _+ - \Phi _- (f) \, J _-, \end{equation} 
with 
\begin{equation} \label{Phi} 
\Phi _+ = \frac{1}{4} \left(\frac{H' (f)}{F' (f)} - 
\frac{H (f)}{f \, F' (f)}\right), \quad \Phi _- = \frac{1}{4} 
\left(\frac{H' (f)}{F' (f)} + \frac{H (f)}{f \, F' (f)}\right). \end{equation}
Since $K$ is Killing with respect to $g_+$, it follows from the Bochner formula that 
\begin{equation} \label{RicK} {\rm Ric} ^{g_+} (K) = \delta \nabla ^{g_+} K, 
\end{equation}
whereas, from (\ref{nabla+K}) we get
\begin{equation} \label{nabla2K} \begin{split} \left(\nabla ^{g_+} \right) ^2 _{X, Y} K &= 
\Phi _+ ' \, df (X) \, J_+ (Y) - \Phi _- ' \, df (X) \, J_- (Y) \\ & - \Phi _- \, \left(\nabla ^{g_+} _X J _-\right) (Y), \end{split} \end{equation}
and, from $\nabla _X ^{g_+} J_- = [J _-, \frac{df}{f} \wedge X]$:
\begin{equation} \label{deltaJ+} \delta J _- = - 
\left(\sum _{i = 1} ^4 \nabla ^{g_+} _{e_i} J_-\right) (e_i) = - 2 J _+ \frac{df}{f}  = - \frac{2}{f F' (f)} \, K ^{\flat}. \end{equation}
By putting together (\ref{nabla+K}), (\ref{RicK}), (\ref{nabla2K}) and (\ref{deltaJ+}), we get
\begin{equation} \label{Riclambda} {\rm Ric} ^{g_+} (K) = \mu \, K, \end{equation}
with
\begin{equation} \mu = - \frac{\big(f \, \Phi ' _+ (f) + f \, \Phi ' _- (f) -  2 \, \Phi _-(f)\big)}{f \, F' (f)}. \end{equation}
Since the metric $g _+$ is K\"ahler with respect to $J _+$, in particular is $J_+$-invariant, (\ref{Riclambda}) implies that the two eigenspaces of ${\rm Ric} ^{g_+}$ are the space $\{K, J_+ K\}$  generated by $K$ and $J _+ K$ (where $J _- = J _+$)  and its orthogonal complement, 
$\{K, J_+ K\} ^{\perp}$ (where $J _- = - J _+$). It follows that ${\rm Ric} ^{g _+}$ is both $J_+$- {\it and} $J _-$-invariant. This establishes the part (i) of the proposition (it is similarly shown that ${\rm Ric} ^{g_-}$ is $J _+$- 
and $J_-$-invariant). 

Before proving  part (ii), we first recall the general transformation rules  of the curvature under a  conformal change of the metric. If $g$ and 
$\tilde{g} = \phi ^{-2} \, g$ are two Riemannian metrics in a same conformal class $[g]$ in any $n$-dimensional Riemannian manifold $(M, g)$, $n > 2$, then the scalar curvature, ${\rm Scal} ^{\tilde{g}}$, and the trace-free part, ${\rm Ric}_0 ^{\tilde{g}}$, of $\tilde{g}$ are related to the 
scalar curvature, ${\rm Scal} ^{g}$, and the trace-free part, 
${\rm Ric}_0 ^{g}$, of $g$ by
\begin{equation} \label{sconformal} {\rm Scal} ^{\tilde{g}} = \phi ^2 \, \big({\rm Scal} ^g - 2 (n - 1) \, \phi \, \Delta _g \phi - n (n - 1) \, |d \phi| ^2 _g\big), \end{equation}
and
\begin{equation} \label{rconformal} {\rm Ric} _0 ^{\tilde{g}} = {\rm Ric} ^g - (n - 2) \, \frac{\left(\nabla ^g d \phi\right) _0}{\phi}, \end{equation}
where $\left(\nabla ^g d \phi\right) _0$ is the trace-free part of the Hessian $\nabla ^g d \phi$ of $\phi$ with respect of $g$, cf. {\it e.g.} \cite[Chapter 1, Section J]{besse}.  Applying (\ref{rconformal}) to the conformal pair $(g _-, g _+ = f ^{-2} \, g _-)$, we get
\begin{equation} \label{Ricg+g-} {\rm Ric} _0 ^{g _+} = {\rm Ric} _0 ^{g _-} -  \frac{2 \, \left(\nabla ^{g_-} d f\right) _0}{f}. \end{equation}
Since ${\rm Ric} ^{g _+}$ and ${\rm Ric} ^{g _+}$ are both $J_+$- and $J_-$-invariant,  it follows that $\left(\nabla ^{g_-} d f \right)_0$ is $J _-$-invariant, as well as $\nabla ^{g_-} d f$, since all metrics in $[g_+] = [g_-]$ are $J_+$- and $J_-$-invariant. This means that the vector field ${\rm grad} _{g_-} f$ is $J_-$-holomorphic, hence that $J _- {\rm grad} _{g_-} f$ is Hamiltonian with respect to $\omega _-$, hence Killing with respect to $g _-$; since $J _- {\rm grad} _{g_-} f = \frac{1}{G' (f)} \, K$, we conclude that $G' (f)$ is constant, hence, by using (\ref{Gprime}), that $F (f)$ and $G (f)$ are of the form
\begin{equation} F (f) = \frac{a}{f} + b, \qquad G (f) = a \, f + c, \end{equation}
for a non-zero real constant $a$ and arbitrary real constants $b, c$. This, together with (\ref{taudf}), establishes  part (ii) of the proposition.
\end{proof}

\begin{thm} \label{thm-calabi} Let $(M, g)$ be a connected, oriented $4$-manifold admitting a non-parallel $*$-Killing $2$-form $\psi = \psi _+ + \psi _-$, satisfying the hypothesis of {\rm Proposition \ref{prop-calabi}}, corresponding to {\rm Case (2)} of {\rm Proposition \ref{prop-f}}.  Then, on the dense open set 
$M _0 \setminus K _1 ^{-1} (0)$ the associated ambik\"ahler structure is of Calabi type, with respect to the Killing vector field $K = K _1$, with $\tau (K) = - K$ in the first case of {\rm Proposition \ref{prop-calabi}} and $\tau (K) = K$ in the two remaining cases.

Conversely, let $(g _+, J _+, \omega _+)$, $(g _- = f ^2 \, g _+, J _-, \omega _-)$ be any ambik\"ahler structure of Calabi type  with non-vanishing Killing vector field $K$, defined on some oriented $4$-dimensional manifold $M$. If $\tau (K) = - K$, there exist, up to scaling\footnote{For any positive constant $c$, the pairs $(g, \psi)$ and $(c \, g, c \, \psi)$ induce the same ambik\"ahler structure.},  a unique metric $g$ in the conformal 
class $[g_+] = [g_-]$ and a unique non-parallel  $*$-Killing $2$-form $\psi$  with respect to $g$,   inducing the given ambik\"ahler structure. If $\tau (K) = K$, such a pair $(g, \psi)$ exists and is unique outside the locus $\{f = 1\}$. 
\end{thm} 
\begin{proof} The first  part of the proposition has already been discussed in the preceding part of this section. 
Conversely, let $(g _+, J _+, \omega _+)$, $(g _- = f ^2 \, g _+, J _-, \omega _-)$ be an ambik\"ahler structure of Calabi type, with respect to some non-vanishing Killing vector field $K$, with $\tau (K) = - K$ or $\tau (K) = K$. Then, according to Proposition \ref{prop-K}, $K$ can be chosen equal to
\begin{equation} K = J _+ {\rm grad} _{g_-} f = J _+ {\rm grad} _{g_+} \frac{1}{f}, \end{equation}
if $\tau (K) = - K$, or 
\begin{equation} K = J _+ {\rm grad} _{g_-} f = - J _+ {\rm grad} _{g_+} \frac{1}{f}, \end{equation}
if $\tau (K) = K$. According to Proposition \ref{prop-kappa} and   (\ref{f-}),  if $\tau (K) = - K$, hence $\tau (df) = - df$, the ambik\"ahler structure is then induced by the  metric $g$, in the conformal class 
$[g_+] = [g _-]$, defined by $g  = f _+ ^{-2} \, g _+ = f _- ^{-2} \, g _-$, 
with  
\begin{equation} \label{g-calabi+} f _+ = \frac{c \, f}{1 + f}, \quad f _- = 
\frac{c}{1 + f} = 
 c -  f _+, \end{equation}
for some positive constant $c$, and the $*$-Killing $2$-form $\psi$ defined by
\begin{equation} \label{psi-calabi+} \psi = \frac{f ^3}{(1 +  f) ^3} \, \omega _+ +   \frac{1}{(1 + f) ^3} \, \omega _-. \end{equation}
If  $\tau (K) = K$, hence $\tau (df) = df$, it similarly  follows from 
Proposition \ref{prop-kappa} and  (\ref{f+}) that 
the ambik\"ahler structure is  induced by the metric $g = f _+ ^2 \, g _+ = f _- ^2 \, g_-$, with
\begin{equation} \label{g-calabi-} f _+ = \frac{c \, f}{1 - f}, \quad f _- = \frac{c}{1 - f} = c + f ^+, \end{equation}
for some constant $c$, positive if $f < 1$, negative if $f > 1$, 
and the $*$-Killing $2$-form 
\begin{equation} \label{psi-calabi-} \psi  = \frac{f ^3}{(1 - f) ^3} \, \omega _+ + \frac{1}{(1 - f) ^3} \, \omega _-, \end{equation}
but the pair $(g, \psi)$ is only defined outside the locus $\{f = 1\}$. 
\end{proof}
\begin{rem} \label{rem-k} {\rm Any ambik\"ahler structure $(g _+, J _+, \omega _+)$, $(g _-, J _-, \omega _-)$ generates, up to global scaling, a $1$-parameter family 
of ambik\"ahler structures, parametrized by a non-zero real number $k$, obtained by, say,  fixing the first K\"ahler structure $(g _+, J _+, \omega _+)$ and  substituting $(g _- ^{(k)} = k ^{-2} \, g _- = f _k ^2  \, g _+, J _- ^{(k)} = \epsilon (k)  \, J _-, \omega _- ^{(k)} = \epsilon (k) \, k ^{-2} \, \omega _-)$ to the second one, 
with $\epsilon (k) = \frac{k}{|k|}$ and $f _k = \frac{f}{|k|}$. Assume that the ambik\"ahler structure  $(g _+, J _+, \omega _+)$, $(g _-, J _-, \omega _-)$ is of Calabi type, with $\tau (df) = - df$.  For any $k$ in $\mathbb{R} \setminus \{0\}$, we then have
$\tau ^{(k)} (d f _k) = - \epsilon (k) \, d f _k$, by setting $\tau ^{(k)} = - J _+ J _- ^{(k)} = - J ^{(k)} _- J _+ = \epsilon (k) \, \tau$, whereas, from (\ref{taudf+-}) we infer:
\begin{equation} f _+ ^{(k)} = \frac{f}{|k + f|}, \quad f _- ^{(k)} = \frac{|k|}{|k + f|}, \end{equation}
up to global scaling; the ambik\"ahler structure $(g _+, J _+, \omega _+)$, 
$(g _- ^{(k)}, J _- ^{(k)}, \omega _- ^{(k)})$ is then induced by the pair $(g ^{(k)}, \psi ^{(k)})$, where $g ^{(k)}$ is defined in the conformal class by
\begin{equation} g ^{(k)} = \frac{f ^2}{(k + f) ^2} \, g _+ = \frac{(1 + f) ^2}{(k + f) ^2} \, g,  \end{equation}
and $\psi ^{(k)}$ is the $*$-Killing $2$-form with respect to $g ^{(k)}$ defined by
\begin{equation} \psi ^{(k)} = \frac{f ^3}{|k + f| ^3} \, \omega _+ + \frac{k}{|k + f| ^3} \, \omega _-, \end{equation}
both defined outside the locus $\{f + k = 0\}$.

} \end{rem} 
\begin{rem} {\rm As observed in  \cite[Section 3.1]{ACG1}, any ambik\"ahler structure of Calabi type 
$(g _+, J _+, \omega _+)$, $(g _- = f ^2 \, g _+, J _-, \omega _-)$, with 
$\tau (df) = df$, admits a {\it Hamiltonian $2$-form}, $\phi ^+$, with respect to the K\"ahler structure $(g _+, J _+, \omega _+)$ and 
a Hamiltonian $2$-form, $\phi ^-$, with respect to the $(g _-, J _-, \omega _-)$, given by
\begin{equation} \label{phi+-} \phi ^+ = f ^{-1} \, \omega _+ 
+ f ^{-3} \, \omega _-, \qquad \phi ^- = f ^3 \, \omega _+ + f \, \omega _-. \end{equation}
} \end{rem}

\section{The decomposable case} \label{sdecomposable}

Assume now that $(M,g,\psi)$ is as in Case (3) in Proposition \ref{prop-f}, that is, 
that the $*$-Killing $2$-form $\psi = \psi _+ + \psi _-$ is degenerate (or decomposable).
This latter condition holds if and only if $\psi \wedge \psi = 0$, if and only if $|\psi _+| = |\psi _-|$, i.e. $f_+ = f _- =: \varphi$, or $f = 1$, meaning that $g _+ = g _- =: g _K$, whereas 
$g = \varphi ^2 \, g _K$.  Denote by $\nabla ^K$ the Levi-Civita connection of $g _K$. Then from \eqref{nJ+}--\eqref{nJ-} we get $\nabla ^K J _+ = \nabla ^K J _- = \nabla ^K \tau = 0$, which implies that $(M, g _K)$ is locally a K\"ahler product of two K\"ahler curves of the form
$M = (\Sigma, g _{\Sigma}, J _{\Sigma}, \omega _{\Sigma}) \times 
(\tilde{\Sigma}, g _{\tilde{\Sigma}}, J _{\tilde{\Sigma}}, \omega _{\tilde{\Sigma}})$, with 
\begin{equation} \begin{split} 
& g _K = g _{\Sigma} + g _{\tilde{\Sigma}}, \\ & 
 J _+ = J _{\Sigma} + J _{\tilde{\Sigma}}, \quad 
J _- = J _{\Sigma} - J _{\tilde{\Sigma}}, \\ & 
 \omega _+ = \omega _{\Sigma} + \omega _{\tilde{\Sigma}}, \quad 
\omega _- = \omega _{\Sigma} - \omega _{\tilde{\Sigma}}.
\end{split} \end{equation}
Moreover, from (\ref{Jdf}) we readily infer $\tau (d \varphi) = d \varphi$, meaning that $\varphi$ is the pull-back to $M$ of a function defined on  $\Sigma$. 
Conversely, for any K\"ahler product $M = (\Sigma, g _{\Sigma}, J _{\Sigma}, \omega _{\Sigma}) \times 
(\tilde{\Sigma}, g _{\tilde{\Sigma}}, J _{\tilde{\Sigma}}, \omega _{\tilde{\Sigma}})$ as above and for {\it any} positive function $\varphi$ defined on $\Sigma$, regarded as a function defined on $M$, the metric $g := \varphi ^2 \, (g_{\Sigma} + g_{\tilde{\Sigma}})$  admits a $*$-Killing $2$-form $\psi$, given by
\begin{equation} \psi = \varphi ^3 \, \omega _{\Sigma}, \end{equation} 
whose corresponding Killing $2$-form $* \psi$ is given by
\begin{equation} * \psi = \varphi ^3 \, \omega _{\tilde{\Sigma}}. \end{equation}
Note that by \eqref{2.2} $\alpha=\frac13\delta^g\psi=\frac1{\varphi^2}*_\Sigma d\varphi$, so $K_1=-\frac12\alpha^\sharp$ is not a Killing vector field in general.

The above considerations completely describe the local structure of 4-manifolds with decomposable $*$-Killing 2-forms. They also provide compact examples, simply by taking $\Sigma$ and $\tilde\Sigma$ to be compact Riemann surfaces. We will show, however, that there are compact 4-manifolds with decomposable $*$-Killing 2-forms which are not products of Riemann surfaces (in fact not even of K\"ahler type). They arise as special cases (for $n=4$) of the classification, in \cite{am}, of compact Riemannian manifolds $(M^n,g)$ carrying a Killing vector fields with conformal Killing covariant derivative. 

It turns out that if $\psi$ is a non-trivial $*$-Killing 2-form which can be written as $\psi=d\xi^\flat$ for some Killing vector field $\xi$ on $M$, then either $\psi$ has rank 2 on $M$, or $M$ is Sasakian or has positive constant sectional curvature (Proposition 4.1 and Theorem 5.1 in \cite{am}). For $n=4$, the Sasakian situation does not occur, and the case when $M$ has constant sectional curvature will be treated in detail in the next section. The remaining case --- when $\psi$ is decomposable --- is the one which we are interested in, and is described by cases 3. and 4. in Theorem 8.9 in \cite{am}. We obtain the following two classes of examples:
\begin{enumerate}
\item $(M,g)$ is a warped mapping torus 
$$M=(\mathbb{R}\times N)/_{(t,x)\sim(t+1,\varphi(x))},\qquad g=\lambda ^2d\theta ^2+g_N,$$
 where $(N,g_N)$ is is a compact
$3$--dimensional Riemannian manifold carrying a function $\lambda$, such that $d\lambda^\sharp$ is a conformal vector field, $\varphi$ is an isometry of $N$ preserving $\lambda$, $\xi =\frac{\partial}{\partial\theta}$ and $\psi=d\xi^\flat=2\lambda d\lambda\wedge d\theta$. One can take for instance $(N,g_N)=\mathbb{S}^3$ and $\lambda$ a first spherical harmonic. Further examples of manifolds $N$ with this property are given in Section 7 in \cite{am}.
\item $(M,g)$ is a Riemannian join $\mathbb{S}^{2}*_{\gamma,\lambda} \mathbb{S}^{1}$, defined as the smooth extension to $S^4$ of the metric 
$g=ds^2+\gamma^2(s){g_{\mathbb{S}^{2}}}+\lambda^2(s)d\theta ^2$ on 
$(0,l)\times {S}^2\times {S}^{1}$, where $l>0$ is a positive real number, $\gamma:(0,l)\to \mathbb{R}^+$ is a smooth function satisfying the boundary
conditions 
$$\gamma(t)=t(1+t^2a(t^2))\quad\hbox{and}\quad \gamma(l-t)=\frac{1}{c}+t^2b(t^2),\quad
\forall\  |t|<\epsilon,$$
for some smooth functions $a$ and $b$ defined on some interval $(-\epsilon,\epsilon)$, $\lambda(s):=\int_s^l \gamma(t)dt$,
$\xi=\frac{\partial}{\partial \theta}$ and $\psi=2\lambda(s)\lambda'(s) ds\wedge d\theta$. 
\end{enumerate}
In particular, we obtain infinite-dimensional families of metrics on 
$S ^3\times S ^1$ and on $S ^4$ carrying decomposable $*$-Killing 2-forms.

\section{Example: the sphere ${\mathbb S} ^4$ and its deformations} \label{ssphere}

We denote by ${\mathbb S} ^4:=(S^4,g)$ the $4$-dimensional sphere, embedded in the standard way in the Euclidean space $\mathbb{R} ^5$, equipped with the standard induced Riemannian metric, $g$, of constant sectional curvature $1$, namely the restriction to $\mathbb{S} ^4$ of the standard inner product $(\cdot, \cdot)$ of $\mathbb{R} ^5$.  We first recall the following well-known facts, cf. e.g. \cite{uwe}. Let $\psi = \psi _+ + \psi _-$ be any $*$-Killing $2$-form with respect to $g$, so that $\nabla _X \Psi = \alpha \wedge X$, cf. (\ref{nablapsi}). Since $g$ is Einstein, the vector field $\alpha ^{\sharp}$ is Killing and it follows from (\ref{nablaalphasharp})--(\ref{h}) that 
$\nabla \alpha = \psi$.  Conversely, for any Killing vector field $Z$ on 
${\mathbb S }^4$, it readily follows from the general {\it Kostant formula}
\begin{equation} \label{kostant} \nabla _X (\nabla Z) = {\rm R} _{Z, X}, \end{equation}
that, in the current case,  $\nabla _X (\nabla Z) = Z \wedge X$, so that the $2$-form $\psi := \nabla Z ^{\flat}$ is $*$-Killing with respect to $g$. The map 
$Z \mapsto \nabla Z ^{\flat}$ is then an isomorphism from the space of Killing vector fields on ${\mathbb S }^4$ to the space of $*$-Killing $2$-forms. 

It is also well-known that there is a natural $1-1$-correspondence between the Lie algebra $\mathfrak{so} (5)$ of anti-symmetric endomorphisms of $\mathbb{R} ^5$ and the space of Killing vector fields on ${\mathbb S} ^4$: for any ${\sf a}$ in $\mathfrak{so} (5)$, the corresponding Killing vector field, $Z _{\sf a}$, is defined by 
\begin{equation} Z _{\sf a} (u) = {\sf a} (u), \end{equation}
for any $u$ in ${\mathbb S} ^4$, where ${\sf a} (u)$ is viewed as an element of the tangent space $T _u {\mathbb S} ^4$, via the natural identification $T _u {\mathbb S} ^4 = u ^{\perp}$.

By combining the above two isomorphisms, we eventually obtained a natural identification of $\mathfrak{so} (5)$ with the space of $*$-Killing $2$-forms on 
${\mathbb S} ^4$ and it is easy to check that, for any ${\sf a}$ in $\mathfrak{so} (5)$, the corresponding  $*$-Killing $2$-form, $\psi _{\sf a}$, is given by
\begin{equation}  \label{psi-A} \psi _{\sf a} (X, Y) = ({\sf a} (X), Y), \end{equation}
for any $u$ in ${\mathbb S} ^4$ and any $X, Y$ in $T _u {\mathbb S} ^4 = u ^{\perp}$; alternatively, the corresponding endomorphism $\Psi _{\sf a}$ is given by
\begin{equation} \label{Psi-A} \Psi _{\sf a} (X) = {\sf a} (X) - ({\sf a} (X), u) \, u, \end{equation}
for any $X$ in $T _u {\mathbb S} ^4 = u ^{\perp}$.

Since, for any $u$ in ${\mathbb S} ^4$,  the volume form of ${\mathbb S} ^4$ is the restriction to 
$T _u {\mathbb S} ^4$ of the $4$-form $u \lrcorner {\rm v} _0$, where ${\rm v} _0$ stands for the standard volume form of $\mathbb{R} ^5$, namely ${\rm v} _0 = e _0 \wedge e _1 \wedge e _2 \wedge e _3 \wedge e _4$, for any direct frame of $\mathbb{R} ^5$ (here identified with a coframe via the standard metric), we easily check that, for any ${\sf a}$ in $\mathfrak{so} (5)$, the corresponding Killing $2$-form $* \psi _{\sf a}$ has the following expression 
\begin{equation} (* \psi _{\sf a}) (X, Y) = (u \lrcorner *_5 {\sf a}) (X, Y) = * _5 (u \wedge {\sf a}) (X, Y), 
 \end{equation}
for any $u$ in ${\mathbb S} ^4$ and any $X, Y$ in $T _u {\mathbb S} ^4 = u ^{\perp}$; here, 
$*_5$ denotes the Hodge operator on $\mathbb{R} ^5$ and we keep identifying vector and covectors via the Euclidean inner product. 

From (\ref{Psi-A}), we easily infer
\begin{equation} |\Psi _{\sf a}| ^2 = |{\sf a}| ^2 - 2 |{\sf a} (u)| ^2, \end{equation}
at any $u$ in ${\mathbb S} ^4$, where the norm is the usual Euclidean norm of endomorphisms, whereas the Pfaffian of $\psi _{\sf a}$ is given by:
\begin{equation} {\rm pf} (\psi _{\sf a}) := \frac{\psi _{\sf a} \wedge \psi _{\sf a}}{2 \, {\rm v} _g} 
 = \frac{(\psi _{\sf a}, * \psi _{\sf a})}{2}  = \frac{u \wedge {\sf a} \wedge {\sf a}}{2 \, {\rm v} _0}. \end{equation}
On the other hand, when $f _+, f _-$ are defined by (\ref{f+-}), we have
\begin{equation} |\Psi _{\sf a}| ^2 = 4 (f _+ ^2 + f _- ^2), \end{equation}
and 
\begin{equation} {\rm pf} (\psi _{\sf a})  = f _+ ^2 - f _- ^2. \end{equation}
For any ${\sf a}$ in $\mathfrak{so} (5)$, we may choose a direct orthonormal basis $e_0, e _1, e_2, e_3, e_4$ of $\mathbb{R} ^5$, with respect to which ${\sf a}$ has the following form
\begin{equation} {\sf a} = \lambda   \, e _1 \wedge e _2 +  \mu  \, e _3 \wedge e _4, \end{equation}
for some real numbers $\lambda, \mu$, with $0 \leq \lambda \leq \mu$.  Then, 
\begin{equation} \begin{split} & |{\sf a}| ^2 = 2 (\lambda  ^2 + \mu ^2), \\ &
{\sf a} (u) = \lambda  (u _1 e _2 - u _2 e_1) + \mu (u _3 e _4 - u _3 e_3), \\ & 
|{\sf a} (u)| ^2 = \lambda  ^2 (u _1 ^2 + u _2 ^2) + \mu  ^2 (u _3 ^2 + u _4 ^2), \\ &  
u \wedge {\sf a} \wedge {\sf a} = 2 \, \lambda  \, \mu \, u _0 \, e _0 \wedge e_1 \wedge e_2 \wedge e _3 \wedge e _4, \end{split} \end{equation}
 for any $u = \sum _{i = 0} ^4 u _i \, e _i$ in $\mathbb{S} ^4$. 
We thus get
\begin{equation} \label{f-sphere} \begin{split} & f _+ ^2 + f _-  ^2 
= \frac{1}{2} \, \big(\lambda  ^2 + \mu ^2 - \lambda  ^2 (u _1 ^2 + u _2 ^2) 
- \mu ^2 (u _3 ^2 + u _4 ^2)\big), 
\\ &
f _+ ^2 -  f _-  ^2 = \lambda  \mu \, u _0, \end{split} \end{equation}
hence
\begin{equation} \label{S4f} \begin{split}
& f _+ (u) = \frac{1}{2} \,  \big((\lambda  + \mu \, u _0) ^2 + (\mu ^2 
- \lambda  ^2) \, (u _1 ^2 + u _2 ^2)\big) ^{\frac{1}{2}} \\ & \quad =
\frac{1}{2} \, \big((\mu + \lambda  \, u _0) ^2 + (\lambda  ^2 - \mu ^2) \, (u _3 ^2 + u _4 ^2)\big) ^{\frac{1}{2}}, \\ \\ & f _- (u)  = \frac{1}{2} \,  
\big((\lambda  - \mu \, u _0) ^2 + (\mu ^2 - \lambda  ^2) \, (u _1 ^2 + u _2 ^2)\big) ^{\frac{1}{2}} \\ & \quad = 
\frac{1}{2} \, \big((\mu - \lambda \,  u _0) ^2 + (\lambda  ^2 - \mu ^2) \, (u _3 ^2 + u _4 ^2)\big) ^{\frac{1}{2}}.
\end{split} \end{equation}
From (\ref{f-sphere})--(\ref{S4f}), we easily obtain the following three cases, corresponding, in the same order, to  the three cases listed in  Proposition \ref{prop-f}:
\begin{enumerate}

\item[{\bf Case 1}]: ${\sf a}$ is of rank $4$ --- i.e. $\lambda$ and $\mu$ 
are both non-zero ---    and $\lambda < \mu$.  Then: 
\smallskip

\begin{enumerate} 
\item[(i)]  $f _+ (u) = f _- (u)$ if and only if $u$ belongs to the equatorial sphere ${\mathbb S} ^3$  defined by $u _0 = 0$;
\smallskip

\item[(ii)]  $f _+ (u) = 0$ if and only $u$ belongs to the circle ${\rm C} _+ = 
\{u _0 = - \frac{\lambda}{\mu}, u _1 = u _2 = 0\}$, and we then have $f _- (u) = \frac{\lambda}{2}$;  
\smallskip

\item[(iii)] $f _- (u) = 0$ if and only if $u$ belongs to the circle ${\rm C} _- = \{u _0 = \frac{\lambda}{\mu}, u _1 = u _2 = 0\}$, and we then have $f _+ (u) =  \frac{\lambda}{2}$;

\smallskip

\item[(iv)]   the $2$-form $d f _+ ^2 \wedge d f _- ^2$ is non-zero outside the $2$-spheres ${\rm S} ^2 _+ = \{u _1 = u _2 = 0\}$ and ${\rm S} ^2 _- = \{u _3 = u _4 = 0\}$; this is because
\begin{equation} \begin{split} d f _+ ^2 \wedge d f _- ^2 & = 
\frac{\lambda  \mu (\lambda  ^2 - \mu ^2)}{2}  \, d u _0 \wedge (u _1 \, d u _1 + u _2 \, d u _2) \\ \qquad & = \frac{\lambda  \mu (\mu ^2 - \lambda  ^2)}{2}  \, d u _0 \wedge (u _3 \, d u _3 + u _4 \, d u _4), \end{split} \end{equation}
which readily follows from (\ref{f-sphere}). 
\end{enumerate} 
\smallskip

\item[{\bf Case 2}]:  ${\sf a}$ is of rank $4$   and $\lambda  =  \mu$. Then
\begin{equation} f _+ (u) = \frac{\lambda}{2} (1 + u_0), \quad f _- (u) = 
\frac{\lambda}{2} (1 - u _0); \end{equation}
in particular,
\begin{equation} f _+ + f _- = \lambda;  \end{equation}
moreover, $f _+ (u) = 0$ if and only if $u = - e _0$ and $f _- (u) = 0$ if and only if $u = e _0$. 
\smallskip

\item[{\bf Case 3}]:  ${\sf a}$ is of rank $2$, i.e.  $\lambda  = 0$.
   Then, $f _+ - f _-$ is identically zero and
$f _+ (u)  = f _- (u)$ vanishes if and only if $u$ belongs to the circle ${\rm C} _0 = \{u _0 = u _1 = u _2 = 0\}$.
\end{enumerate}
\begin{rem} {\rm Consider the functions $x = \frac{f _+ + f _-}{2}, y = \frac{f_+ - f _-}{2}$ defined in Section \ref{sorthogonal}, as well as the functions of one variable, $A$ and $B$,  appearing in Proposition \ref{prop-separate}. 
 If ${\sf a}$ is of rank $4$, with $0 < \lambda  <  \mu$, corresponding to   {Case 1} in the above list, we easily infer from (\ref{f-sphere}) that
\begin{equation} \label{xyu} \begin{split} & u_0  = \frac{4 x  y }{\lambda    \mu }, \\ &  u _1 ^2 + u_2 ^2 = \frac{(\lambda ^2 - 4 x ^2) 
(\lambda ^2 - 4 y ^2)}{\lambda ^2 (\lambda ^2 - \mu ^2)}, \\ &  u _3 ^2 + u _4 ^2 =
 \frac{(\mu ^2 - 4 x ^2) (\mu ^2 - 4 y ^2)}{\mu ^2 (\mu ^2 - \lambda ^2)}. \end{split} \end{equation}
Since $x  \geq |y|$, the above identities imply that the image of $(x, y)$ in $\mathbb{R} ^2$ is the rectangle 
$R := \left[\frac{\lambda}{2}, \frac{\mu}{2}\right] \times \left[- \frac{\lambda}{2}, \frac{\lambda}{2}\right]$. A simple calculation then shows that 
$A$ and $B$ are given by
\begin{equation} \label{AB-sphere} A (z) = - B (z) = - 
\left(z ^2 - \frac{\lambda ^2}{4}\right) \left(z ^2 - \frac{\mu ^2}{4}\right).
 \end{equation}
Notice that $A (x)$ and $B (y)$ are positive in the interior of $R$,
 corresponding to the open set of $\mathbb{S} ^4$ where $dx, dy$ are independent,  and vanish on its  boundary. 
Also notice that the above expressions of $A, B$ fit  with the identities (\ref{scal-AB})--(\ref{b-AB}), with ${\rm Scal} = 12$ and $b = 0$. 
} \end{rem}

\begin{rem} \label{rem-deformation}  {\rm By using the ambitoric Ansatz in  Theorem \ref{thm-ambitoric}, the above situation can easily be deformed in Case 1, where  ${\sf a}$ is of rank $4$, with $0 < \lambda  <  \mu$,    and the $2$-form $\psi _{\sf a}$ defined by \eqref{psi-A} is $*$-Killing with respect to the round metric\footnote{We warmly thank Vestislav Apostolov for this suggestion.}. On the open set $\mathcal U = \mathbb{S} ^4 \setminus \left(S ^2 _+ \cup S ^2 _-\right)$, where $f_+\ne 0$, $f_-\ne 0$ and $df_+\wedge df_-\ne 0$, the round metric of $\mathbb{S} ^4$ takes the form \eqref{A-g}, where $A$ and $B$ are given by \eqref{AB-sphere}, $x\in \left(\frac{\lambda}{2}, \frac{\mu}{2}\right)$, $y\in\left( - \frac{\lambda}{2}, \frac{\lambda}{2}\right)$ are determined by \eqref{xyu} and $ds$, $dt$ are explicit exact  1-forms determined by the last two equations of \eqref{A-J}\footnote{It can actually be  shown that outside the $2$-spheres $S ^2 _+$ and $S ^2 _-$, $d s$ and $d t$ are given by:
\begin{equation} \begin{split} 
& d s = \frac{2}{\mu ^2 - \lambda ^2} \, \left(\lambda \, \frac{u _1 d u _2 - u _2 d u _1}{u _1 ^2 + u _2 ^2} - \mu \, \frac{u _3 d u _4 - u _4 d u _3}{u _3 ^2 + u _4 ^2}\right) = \frac{2}{\mu ^2 - \lambda ^2} \, d \left(\lambda \, \arctan{\frac{u _2}{u _1}} - \mu \, \arctan{\frac{u _4}{u _3}}\right), \\ &
d t = \frac{8}{\mu ^2 - \lambda ^2} \, \left(- \frac{1}{\lambda} \, \frac{u _1 d u _2 - u _2 d u _1}{u _1 ^2 + u _2 ^2} + \frac{1}{\mu}  \, \frac{u _3 d u _4 - u _4 d u _3}{u _3 ^2 + u _4 ^2}\right) = \frac{8}{\mu ^2 - \lambda ^2} \, d \left(- \frac{1}{\lambda} \, \arctan{\frac{u _2}{u _1}} +  \frac{1}{\mu} \, \arctan{\frac{u _4}{u _3}}\right). \end{split} \end{equation}}. Moreover, 
$\psi _{\sf a}$ is given by \eqref{A-psi} with respect to these coordinates.

Consider now a small perturbation $\tilde A$, $\tilde B$ of the functions $A$ and $B$ such that $\tilde A(x)=A(x)$ near $x=\frac{\lambda}{2}$ and $x= \frac{\mu}{2}$ and $\tilde B(y)=B(y)$ near $y=\pm\frac{\lambda}{2}$. If the perturbation is small enough, the expression analogue to \eqref{A-g}
\begin{equation} \label{A-tildeg} \begin{split} \tilde g & := (x ^2 - y ^2) \, \left(\frac{d x \otimes d x}{\tilde A (x)} + \frac{d y \otimes d y}{\tilde B (y)}\right) \\ & \quad 
+ \frac{\tilde A (x)}{(x ^2 - y ^2)} \, (d s + y ^2 \, dt) \otimes 
(d s +  \, y ^2 \, dt) \\ & \quad 
+ \frac{\tilde B (y)}{(x ^2 - y ^2)} \, (d s + x ^2 \, dt) \otimes 
(d s + x ^2 \, dt) \end{split} \end{equation}
is still positive definite so defines a Riemannian metric on $\mathcal U$, which coincides with the canonical metric on an open neighbourhood of 
$\mathbb{S} ^4 \setminus \mathcal{U} = S ^2 _+ \cup S ^2 _-$, and thus has a smooth extension to $\mathbb{S} ^4$ which we still call $\tilde g$. Since the expression \eqref{A-psi} of the $*$-Killing form in the Ansatz of Section \ref{sambitoric} does not depend on $A$ and $B$, the 2-form $\psi _{\sf a}$ is still $*$-Killing with respect to the new metric $\tilde g$. We thus get an infinite-dimensional family (depending on two functions of one variable) of Riemannian metrics on $S^4$ which all carry {\em the same} non-parallel $*$-Killing form. 
} \end{rem}

\section{Example: complex ruled surfaces} \label{shirzebruch}

In general, a (geometric) {\it complex ruled surface} is a compact, connected, complex manifold of the form $M = \mathbb{P} (E)$, where $E$ denotes a rank 2 holomorphic vector bundle over some  (compact, connected)  Riemann surface, $\Sigma$, and $\mathbb{P} (E)$ is then the corresponding projective line bundle, i.e. the holomorphic bundle over $\Sigma$, whose fiber at each point $y$ of $\Sigma$ is the complex projective line $\mathbb{P} (E _y)$, where $E _y$ denotes the fiber of $E$ at $y$.  A complex ruled surface is said to be {\it of genus} ${\bf g}$ 
if $\Sigma$ is of genus ${\bf g}$.

In this section,  we restrict our attention to complex ruled surfaces $\mathbb{P} (E)$ as above, when  $E = L \oplus \mathbb{C}$ is the Whitney sum of some holomorphic line bundle, $L$, over $\Sigma$ and of the trivial complex line bundle $\Sigma \times \mathbb{C}$, here simply denoted $\mathbb{C}$: $M$ is then the {\it compactification}  of the total space of $L$ obtained by adding the  {\it point at infinity} $[L _y] := \mathbb{P} (L _y \oplus \{0\})$ to  each fiber of $M$ over $y$. The union of the points at infinity is a divisor of $M$, denoted by $\Sigma _{\infty}$, whereas the (image of) the zero section of $L$, viewed as a divisor of $M$, is denoted $\Sigma _0$; both $\Sigma _0$ and $\Sigma _{\infty}$ are identified with $\Sigma$ by the natural projection, $\pi$, from $M$ to $\Sigma$. The open set $M \setminus (\Sigma _0 \cup \Sigma _{\infty})$, denoted $M ^0$, is naturally identified with $L \setminus \Sigma _0$.  We moreover assume that the degree, ${\rm d} (L)$,  of $L$ is {\it negative} and we set: 
${\rm d} (L) = - k$, 
where $k$ is a positive integer.

Complex ruled surfaces of this form will be called {\it Hirzebruch-like ruled surfaces}. When  ${\bf g} = 0$, these are exactly those complex ruled surfaces introduced by F. Hirzebruch in \cite{hirzebruch}. When ${\bf g} \geq 2$, they were named {\it pseudo-Hirzebruch} in \cite{christina}. 

In general, the K\"ahler cone of a complex ruled surface $\mathbb{P} (E)$ was described by A.  Fujiki in \cite{fujiki}. In the special case considered in this section, when $M = \mathbb{P} (L \oplus \mathbb{C})$ is a Hirzebruch-like ruled surface,   if $[\Sigma _0]$, $[\Sigma _{\infty}]$ and $[F]$ denote the Poincar\'e duals of the (homology class of) $\Sigma _0$, $\Sigma _{\infty}$ and of any fiber $F$ of $\pi$ in ${\rm H ^2} (M, \mathbb{Z})$, the latter is freely generated by $[\Sigma _0]$ and $[F]$ or by $[\Sigma _{\infty}]$ and $[F]$, with
$[\Sigma _0] = [\Sigma _{\infty}] - k \, [F]$, 
and the K\"ahler cone is the set of those elements, $\Omega _{a _0, a _{\infty}}$, of ${\rm H} (M, \mathbb{R})$  which are of the form
$\Omega _{a_0, a _{\infty}} = 2 \pi \, \big(- a _0 \, [\Sigma _0] + a _{\infty} \, [\Sigma _{\infty}]\big)$, 
for any two real numbers $a _0, a _{\infty}$ such that
$ 0 < a _0 < a _{\infty}$.

We assume that $\Sigma$ comes equipped with a K\"ahler metric $(g _{\Sigma}, \omega _{\Sigma})$ polarized by $L$, in the sense that $L$ is endowed with a Hermitian (fiberwise) inner product, $h$, in such a way that the curvature, ${\rm R} ^{\nabla}$,  of the associated Chern connection, $\nabla$, is related to the K\"ahler form $\omega _{\Sigma}$  by ${\rm R} ^{\nabla} = i \, \omega$; in particular, $[\omega _{\Sigma}] = 2 \pi \, c _1 (L ^*)$, where 
$[\omega _{\Sigma}]$ denotes the de Rham class of $\omega _{\Sigma}$, $L ^*$ the dual line bundle to $L$ and $c _1 (L ^*)$ the (de Rham) Chern class of $L ^*$. 
The natural action of $\mathbb{C} ^*$ extends to a holomorphic $\mathbb{C} ^*$-action on $M$, trivial on $\Sigma _0$ and $\Sigma _{\infty}$; we denote by $K$ the generator of the restriction of this action on $S ^1 \subset \mathbb{C} ^*$. 
On $M ^0 = L \setminus \Sigma _0$, we denote by $t$ the function defined by 
\begin{equation} \label{t} t = \log{r}, \end{equation}
where $r$ stands for the distance to the origin in each fiber of $L$ determined by $h$;  on $M ^0$, we then have
\begin{equation} d d ^c t = \pi ^* \omega _{\Sigma}, \qquad d ^c t (K) = 1 \end{equation}
(beware: the function $t$ defined by (\ref{t}) has nothing to do with the local coordinate $t$ appearing in Section \ref{sambitoric}). 
Any (smooth) function $F = F (t)$ of $t$ will be regarded as function defined on $M ^0$, which is evidently $K$-invariant; moreover:
\begin{enumerate} \label{asymptotic} 

\item    $F = F (t)$ smoothly extends to $\Sigma _0$  if and only if
$F (t) = \Phi _+ (e ^{2t})$  near $t = - \infty$, for some smooth function $\Phi _+$ defined on some neighbourhood of $0$ in $\mathbb{R} ^{\geq 0}$,  and

\item  $F = F (t)$ smoothly extends to $\Sigma _{\infty}$ if and only if $F (t) = \Phi _- (e ^{-2t})$ near 
$t = \infty$,  for some smooth function $\Phi _-$ defined   
on some neighbourhood of $0$ in $\mathbb{R} ^{\geq 0}$, cf. {\it e.g.} \cite{christina}, \cite[Section 3.3]{ACG1}.
\end{enumerate}

For any (smooth) real function $\varphi = \varphi (t)$, denote by $\omega _{\varphi}$ the real, $J$-invariant $2$-form defined on $M ^0$ by
\begin{equation} \label{omegaphi} \omega _{\varphi} = \varphi  \, d d ^c t + \varphi '  \, dt \wedge d ^c t, \end{equation}
where $\varphi '$ denotes the derivative of $\varphi$ with respect to $t$. 
Then,  $\omega _{\varphi}$ is a K\"ahler form on $M ^0$, with respect to the natural complex structure $J = J _+$, of $M$,  if and only if $\varphi$ is positive and increasing as a function of $t$; moreover, $\omega _{\phi}$ extends to a smooth K\"ahler form on $M$, in the K\"ahler class $\Omega _{a _0, a _{\infty}}$,  if and only if $\varphi$ satisfies the above asymptotic conditions (1)--(2), with $\Phi _+ (0) = a _0 > 0$, $\Phi _+ ' (0) > 0$, $\Phi _- (0) = a _{\infty} > 0$, $\Phi _- ' (0) < 0$. K\"ahler forms of this form on $M$, as well as the corresponding K\"ahler metrics 
\begin{equation} g _{\varphi} = \varphi  \, \pi ^* g _{\Sigma} + \varphi '  \, (d t \otimes d t + d ^c t \otimes d ^c t), \end{equation}
are called {\it admissible}. 

Denote by $J _-$ the complex structure, first defined on the total space of $L$ by keeping $J$ on the horizontal distribution determined by the Chern connection and by substituting $- J$ on the fibers,  then smoothly extended to $M$. The new complex structure induces the opposite orientation, hence commutes with $J _+ = J$. 

Any admissible K\"ahler form $\omega _{\varphi}$ is both $J _+$- and $J _-$-invariant, as well as the associated $2$-form $\tilde{\omega} _{\varphi}$ defined by
\begin{equation} \tilde{\omega} _{\varphi} := \frac{1}{\varphi} \, d d ^c t - 
\frac{\varphi '}{\varphi ^2} \, d t \wedge d ^c t, \end{equation}
which is moreover K\"ahler with respect to $J _-$, with metric
\begin{equation} \tilde{g} _{\varphi} = \frac{1}{\varphi ^2} \, g _{\varphi}. \end{equation}
We thus obtain an ambik\"ahler structure of Calabi-type, as defined in Section \ref{scalabi}, with $f = \frac{1}{\varphi}$ and $\tau (K) = - K$. According to Theorem \ref{thm-calabi} and Remark \ref{rem-k}, for any $k$ in $\mathbb{R} \setminus \{0\}$, the metric $g ^{(k)}$ defined, outside the locus $\{1 + k \, \varphi = 0\}$,  by
\begin{equation} g ^{(k)}_{\varphi}  = \frac{1}{(1 + k \, \varphi) ^2} \, g _{\varphi}, \end{equation}
there admits a non-parallel $*$-Killing $2$-form $\psi ^{(k)} _{\varphi}$, namely 
\begin{equation} \begin{split} \psi ^{(k)}  _{\varphi} & = \frac{1}{(1 + k \, \varphi) ^3} \, \omega _{\phi} + \frac{k \, \varphi ^3}{(1 +  k \, \varphi) ^3} \, \tilde{\omega} _{\varphi} \\ & = \frac{\varphi}{(1 + k \, \varphi) ^2} \, d d ^c t + \frac{(1 - k \, \varphi) \varphi '}{(1 +  k \, \varphi) ^3} \, d t \wedge d ^c t. \end{split} \end{equation}
Notice that the pair $(g ^{(k)} _{\varphi}, \psi ^{(k)} _{\varphi})$ smoothly extends to $M$ 
for any $k\in\mathbb{R}\setminus[- \frac{1}{a _0}, - \frac{1}{a _{\infty}}]$, including $k = 0$ for which  we simply get the K\"ahler pair $(g _{\varphi}, \omega _+)$.

\end{document}